\documentclass[a4paper,11pt,leqno]{amsart}

\usepackage{vmargin}

\setmarginsrb{25mm}{25mm}{25mm}{20mm}{0pt}{0pt}{0pt}{12mm}

\newtheorem{theorem}{Theorem}[section]

\newtheorem{lemma}[theorem]{Lemma}

\newtheorem{remark}[theorem]{Remark}

\newcommand{\R}{\mathbb R}
\newcommand{\N}{\mathbb N}
\newcommand{\Z}{\mathbb Z}

\newcommand{\HE}{\mathbb H}

\newcommand{\A}{\mathcal A}
\newcommand{\V}{\mathcal V}
\newcommand{\ZZ}{\mathcal Z}
\newcommand{\G}{\mathcal G}
\newcommand{\ST}{\mathcal S}
\newcommand{\F}{\mathfrak F}

\newcommand{\T}{\mathcal T}
\newcommand{\CH}{\mathcal C}
\newcommand{\Hk}{\mathcal{H}^k}
\newcommand{\kulma}[2]{\angle({#1},{#2})}

\DeclareMathOperator{\spt}{spt}
\DeclareMathOperator{\sisus}{int}
\DeclareMathOperator{\Lip}{Lip}
\newcommand{\LEK}{\mathcal{L}^k}

\def\Xint#1{\mathchoice
   {\XXint\displaystyle\textstyle{#1}}%
   {\XXint\textstyle\scriptstyle{#1}}%
   {\XXint\scriptstyle\scriptscriptstyle{#1}}%
   {\XXint\scriptscriptstyle\scriptscriptstyle{#1}}%
   \!\int}
\def\XXint#1#2#3{{\setbox0=\hbox{$#1{#2#3}{\int}$}
     \vcenter{\hbox{$#2#3$}}\kern-.5\wd0}}
\def\dashint{\Xint-}

\begin{document}

\title[A sufficient condition]{A sufficient condition for having big pieces of bilipschitz images of subsets of Euclidean space in Heisenberg groups}
\author{Immo Hahlomaa}
\address{Department of Mathematics and Statistics, P.O. Box 35 (MaD), FI-40014 University of Jy\-v\"as\-ky\-l\"a, Finland}
\email{immo.a.hahlomaa@jyu.fi}
\thanks{The author acknowledges the support of the Academy of Finland, project \#131729}
\date{}
\pagestyle{plain}

\begin{abstract}
In this article we extend a euclidean result of David and Semmes to the Heisenberg group by giving
a sufficient condition for a $k$-Ahlfors-regular subset to have
big pieces of bilipschitz images of subsets of $\R^k$.
This Carleson type condition measures how well the set can be approximated by the Heisenberg $k$-planes at different scales and locations.
The proof given here follow the paper of David and Semmes.
\end{abstract}

\maketitle

\section{Introduction}

In \cite{MR1069238} and \cite{MR1182488} Jones and Okikiolu proved that a bounded set $E\subset\R^n$ is contained in a rectifiable curve
if and only if
\begin{align}
\label{BJ}
  \int_0^\infty\int_{\R^n} \beta^E_\infty(x,t)^2\, dx\, \frac{dt}{t^n} < \infty,
\end{align}
where
\[ \beta^E_\infty(x,t) = \inf_L t^{-1}\sup\left\{\, d_e(y,L)\, :\, y\in B^{d_e}_E(x,t)\, \right\}, \]
with the infimum taken over all lines in $\R^n$.
Here and in the sequel $d_e$ denotes the euclidean metric in $\R^n$ (for any $n$ in question each time)
and $B^\rho_Y(x,r) = B_Y(x,r) = \{ y\in X\, :\, \rho(y,x) \leq r \}$ 
for a metric space $(X,\rho)$, $Y\subset X$, $x\in X$ and $r\geq 0$.
In \cite{MR1113517} David and Semmes gave a higher dimensional version of the above theorem for $k$-regular subsets of $\R^n$
(where $k$ is an integer between $0$ and $n$)
by showing that a closed $k$-regular set $E\subset\R^n$ has big pieces of Lipschitz images of $\R^k$ if and only if
there is $C<\infty$ such that
\begin{align}
\label{urehto}
  \int_0^r\int_{B_E(z,r)} \beta^E_1(x,t)^2\, d\mathcal{H}^k_E(x)\, \frac{dt}{t} \leq Cr^k
\end{align}
for all $z\in E$ and $r>0$, where
\begin{align*}
  \beta^E_1(x,t) = t^{-k-1}\inf_L\int_{B_E(x,t)} d_e(y,L)\, d\mathcal{H}^k_E(y)
\end{align*}
with infimum taken over all $k$-planes in $\R^n$.
In fact David and Semmes gave in \cite{MR1113517} several equivalent conditions to \eqref{urehto} and said that
a closed $k$-regular set $E\subset\R^n$ is uniformly rectifiable if it satisfies these conditions.

Above the metric notions are of course taken with respect to $d_e$. More generally, we say that
a metric space $(X,\rho)$ is $k$-\emph{regular} if there exists a constant $C\in\R$ such that
$C^{-1}r^k \leq \Hk_X(B^\rho_X(x,r)) \leq Cr^k$ for any $x\in X$ and $r\in ]0,\rho(X)]$,
where $\Hk_X$ is the $k$-dimensional Hausdorff measure on $(X,\rho)$.
The smallest such constant $C$ will be denoted by $C_{(X,\rho)}$.
Further we say that $(X,\rho)$ \emph{has big pieces of bilipschitz images of subsets of} $\R^k$ (with constants $K$ and $c$)
if for any $x\in X$ and $r\in ]0,\rho(X)[$ there exists a $K$-bilipschitz function $f:A\to X$
(w.r.t. the metrics $d_e$ and $\rho$) with
$A\subset B^{d_e}_{\R^k}(0,r)$ such that $\Hk_X(f(A)\cap B_X(x,r))\geq cr^k$. 

In \cite{MR2373273} Schul extended the one dimensional result of Jones and Okikiolu to Hilbert spaces (with the condition \eqref{BJ} modified in an appropriate way). 
Further in \cite{MR2371434} Ferrari, Franchi and Pajot gave for a compact subset $E$ of the Heisenberg group $\HE^1$
(endowed with its Carnot-Carath\'eodory metric)
an analogue of the condition~\eqref{BJ}
which measures the deviation of $E$ from a best approximating Heisenberg straight line
(i.e. an element of $\V^1$, see \eqref{V}) at different scales and locations.
They showed that this condition is sufficient for $E$ to be  contained in a rectifiable curve.
Juillet gave in \cite{MR2789375} an example which shows that it is not necessary.
Following \cite{MR2371434} we define for a $k$-regular subset $E$ of the Heisenberg group $\HE^n$ an analogue for \eqref{urehto}
and give a proof for the following theorem.
For the definitions see Sections~\ref{secH} and \ref{secbd}.
Theorem~\ref{th} is invariant under bilipschitz change of metric (see also \eqref{beta1}).
Note that the Kor\'anyi metric $d$ (see \eqref{d}) which we below use exclusively is bilipschitz equivalent with the
usual Carnot-Carath\'eodory metric on $\HE^n$.

\begin{theorem}
\label{th}
Let $k\in\N$ and $E$ be a $k$-regular subset of the Heisenberg group $\HE^n$.
If there is a constant $C$ such that
\begin{align}
\label{oletus}
  \int_0^r\int_{B_E(x,r)} \beta^E_1(y,t)^2\, d\Hk_E(y)\, \frac{dt}{t} \leq Cr^k\qquad\text{for all $x \in E$ and $r>0$},
\end{align}
then $E$ has big pieces of bilipschitz images of subsets of $\R^k$.
\end{theorem}

The proof given here follows \cite{MR1113517}. A similar method is applied also in \cite{MR1709304}.
For readability and consistency we give a quite detailed proof although mostly the adaptation from \cite{MR1113517} is trivial or at least straightforward.

In this article $|\cdot|$ denotes the euclidean $k$-norm for any $k$ in question.
The cardinality of a finite set $X$ is denoted by $\#X$.
Further $\mathcal{P}(X)=\{ Y\, :\, Y\subset X \}$ for any set $X$, and the symbol $\dashint$ is used to denote an average integral.

\section{Some notations and preliminaries on Heisenberg groups}
\label{secH}

The Heisenberg group $\HE^n$ is the unique simply connected and connected Lie group of step two and dimension $2n+1$
with one dimensional center. As a set it may be identified with $\R^{2n+1}$. The points $x \in \HE^n$ are written as 
$x = (x',x_{2n+1})$ with $x' \in \R^{2n}$ and $x_{2n+1} \in \R$. The group operation is given by
\[ x \cdot y = ( x' + y' , x_{2n+1} + y_{2n+1} + 2A(x',y')), \]
where
\[ A(x',y') = \sum_{i=1}^n ( x_{i+n}y_i - x_iy_{i+n} ). \]
Note that the inverse of $x$, denoted also by $x^{-1}$, is $-x = (-x',-x_{2n+1})$ and the neutral element is $(0,0)$.
We equip $\HE^n$ by a metric $d$ defined by
\begin{align}
\label{d}
  d(x,y) = \lVert y^ {-1} \cdot x \rVert,\qquad\text{where $\lVert x \rVert = (|x'|^4 + x_{2n+1}^2)^{1/4}$}.
\end{align}
The metric $d$ is left invariant i.e. for each $p\in\HE^ n$ the left translation $\tau_p:x\mapsto p\cdot x$ is
an isometry from $(\HE^n,d)$ to itself.
Note that for $x,y\in\R^{2n} \times \{ 0 \}$ the conditions
$A(x',y') = 0$, $x\cdot y = x + y$ and $d(x,y) = |x-y|$
are equivalent.

A linear subspace $V\subset\R^{2n}$ is said to be \emph{isotropic} if $A(x,y)=0$ for all $x,y\in V$.
For $k \in \N$ denote
\begin{align*}
  \V_0^k = \left\{\, V\times\{ 0\} \, :\, \text{$V$ is an isotropic $k$-dimensional linear subspace of $\R^{2n}$}\, \right\}.
\end{align*}
In other words $\V_0^k$ is the collection of the $k$-dimensional homogenous horizontal subgroups of $\HE^n$ (see \cite{MR2955184}).
We note that $\V_0^k \neq \emptyset$ if and only if $0 \leq k \leq n$.
Set
\begin{align}
\label{V}
  \V^k = \left\{\, \tau_p(V)\, :\, \text{$V \in \V_0^ k$ and $p \in \HE^n$}\, \right\}. 
\end{align}
Each $V \in \V^k$ is a $k$-dimensional affine subspace of $\R^{2n+1}$ because $\tau_p$ is an affine mapping whose linear part has determinant $1$.
Note also that for any $V\in\V^k$
\begin{align}
\label{dV}
  d(x,y) = |x'-y'|\qquad\text{for all $x,y\in V$}.
\end{align}
Namely, let $V=\tau_p(V_0)$ for $V_0\in\V_0^k$, $x=p\cdot x_0$ and $y=p\cdot y_0$. Since $(p\cdot z)' = p' + z'$ for any $z\in\HE^n$,
one has $d(x,y) = d(x_0,y_0) = |x_0'-y_0'| = |x'-y'|$.

For $V \in \V^k$ we define the projection $P_V : \HE^n \to V$ by setting
\begin{align*}
   P_V = \tau_{p} \circ P^e_{\tau_{-p}(V)} \circ \tau_{-p},
\end{align*}
where $p \in V$ and $P^e_L:\HE^n\to L$ is the euclidean orthogonal projection to the linear subspace $L \in \V^k_0$ (called \emph{horizontal projection} in \cite{MR2789472}).
Note that the definition of $P_V$ is correct, because letting $V_0 \in \V_0^k$ and $x\in\HE^n$
\[ P^e_{V_0}(a\cdot x) = P^e_{V_0}(a'+x',0) = P^e_{V_0}(a',0) + P^e_{V_0}(x',0) = P^e_{V_0}(a) \cdot P^e_{V_0}(x) \]
for every $a\in\HE^n$, and hence
\begin{align*}
  p \cdot v \cdot P^e_{V_0}((p\cdot v)^{-1} \cdot x) = p \cdot v \cdot (P^e_{V_0}(-v) \cdot P^e_{V_0}(p^{-1} \cdot x))
  = p \cdot P^e_{V_0}(p^{-1}\cdot x).
\end{align*}
for any $p\in\HE^n$ and $v\in V_0$.
It is easy to see (using the left invariance of $d$) that $P_V$ is 1-Lipschitz for any $V\in\V^k$.

\begin{remark}\rm
\label{remProttr}
A linear map $\varphi:\HE^n\to\HE^n$ is called a \emph{rotation}
if $\varphi(x)_{2n+1} = x_{2n+1}$, $A(x',y')=A(\varphi(x)',\varphi(y)')$ and $|\varphi(x)'-\varphi(y)'| =|x'-y'|$ for all $x$ and $y$.
If $\varphi$ is a rotation then clearly $\varphi(x\cdot y)=\varphi(x)\cdot\varphi(y)$
and $\varphi(V) \in \V^k$ for any for any $x,y\in\HE^n$ and $V\in\V^k$. Hence
\[ P_{\tau_p\circ\varphi(V)}\circ\tau_p\circ\varphi = \tau_p\circ\varphi\circ P_V \]
for any $p\in\HE^n$, $V\in\V^k$ and rotation $\varphi$.
Namely, if $V=\tau_q(V_0)$ for $q\in\HE^n$ and $V_0\in\V_0^k$ then for $x\in\HE^n$
\begin{align*}
  &P_{\tau_p\circ\varphi(V)}(p\cdot\varphi(x)) = P_{p\cdot\varphi(q)\cdot\varphi(V_0)}(p\cdot\varphi(x)) \\
  &= p\cdot\varphi(q)\cdot P^e_{\varphi(V_0)}((p\cdot\varphi(q))^{-1}\cdot p\cdot\varphi(x))
  = p\cdot\varphi(q)\cdot P^e_{\varphi(V_0)}(\varphi(x)-\varphi(q))  \\
  &= p\cdot\varphi(q)\cdot \varphi(P^e_{V_0}(x-q))
  = p\cdot\varphi(q\cdot P^e_{V_0}(x-q))
  = p\cdot\varphi(P_V(x)).
\end{align*}
For any $V_0,W_0\in\V_0^k$ there is a rotation $\varphi$ such that $W_0=\varphi(V_0)$ (see \cite{MR2955184}).
Hence for any $V,W\in\V^k$ there is a rotation~$\varphi$ and $p\in\HE^n$ such that $W=\tau_p\circ\varphi(V)$.
Notice that the rotations are isometries.
\end{remark}

Denote $X_k = \{\, x\in\HE^n\, :\, \text{$x_i=0$ for all $i>k$}\, \}$.

\begin{lemma}
\label{lePdist}
For any $x\in\HE^n$ and $V\in\V^k$
\[ d(x,P_V(x)) \leq 3d(x,V). \]
\end{lemma}

\begin{proof}
By the left invariance of $d$ this follows  from \cite{MR2789472} (at least with $3$ replaced by some constant).
Let us give here another proof by a direct calculation.
By \ref{remProttr} one only needs show that
$d(x,P_{X_k}(x)) \leq 3d(x,X_k)$ for all $x\in\HE^n$. Let $(z,y)\in\R^k\times\R^{2n-k}$, $u\in X_k$ and $C>1$.
Define $f:\R\to\R$ by setting
$f(t) = Cd((z,y,t),u)^4 - d((z,y,t),P_{X_k}(z,y,t))^4$. Let $t\in\R$. By denoting
$T=-2\sum_{i=1}^ku_iy_{n-k+i}$ and $U=-2\sum_{i=1}^kz_iy_{n-k+i}$ we have
\[ f(t) = C\left(|z-u|^2+|y|^2\right)^2+C(t-T)^2-|y|^4-(t-U)^2. \]
Now
\begin{align*}
  f(t) \geq f\left(\frac{CT-U}{C-1}\right) = C\left(|z-u|^2+|y|^2\right)^2 - \frac{C(T-U)^2}{C-1} - |y|^4.
\end{align*}
Since
\begin{align*}
  |T-U| = 2\left|\sum_{i=1}^k(z_i-u_i)y_{n-k+i}\right| \leq 2|z-u||y| \leq \left(|z-u|^2+|y|^2\right), 
\end{align*}
we have
\begin{align*}
  f(t) \geq \frac{C(C-2)}{C-1}\left(|z-u|^2+|y|^2\right)^2-|y|^4 \geq 0
\end{align*}
by choosing $C\geq 3$.
\end{proof}

For $V,W\in\V^k$ denote
\[ \kulma{V}{W} = \min\{\, C \geq 1\, : \text{$d(x,y) \leq Cd(P_W(x),P_W(y))$ for all $x,y\in V$}\, \}. \]
Let $P^e_L$ denote also the euclidean orthogonal projection from $\R^{2n}$ to an affine subspace $L\subset\R^{2n}$.
For any $Y\subset\HE^n$ we write $Y'=\{ x'\, :\, x \in Y \}$. 

Let $V=\tau_p(V_0)$ for $p\in\HE^n$ and $V_0\in\V_0^k$. Then  and
\[ P_V(x)'=(p\cdot P^e_{V_0}(x-p))' = p'+P^e_{V_0}(x-p)' = p'-P^e_{V_0}(p)'+P^e_{V_0}(x)' \]
for any $x$ (by the linearity of $P^e_{V_0})$.
Thus, since $V' = p' + V_0'$, we have by \eqref{dV}
\begin{align}
\label{PV}
  d(P_V(x),P_V(y))=|P_V(x)'-P_V(y)'|=|P^e_{V_0}(x)'-P^e_{V_0}(y)'|=|P^e_{V'}(x')-P^e_{V'}(y')|
\end{align}  
for any $x,y\in\HE^n$.
Particularly the following equality holds.

\begin{lemma}\rm
\label{lekulmat}
  Let $V,W\in\V^k$. Then
  \[ \kulma{V}{W}=\min\{\, C \geq 1\, : \text{$|x-y| \leq C|P^e_{W'}(x)-P^e_{W'}(y)|$ for all $x,y\in V'$}\, \}. \]
\end{lemma}

\section{Beta numbers and dyadic cubes}
\label{secbd}

Let $k\in\{1,\dotsc,n\}$. From this on we assume that $E$ is a $k$-regular subset of $\HE^n$.
Denote $B(x,r) = B^d_{\HE^n}(x,r)$ and $\mu = \Hk|_E$, where $\Hk$ is the $k$-dimensional Hausdorff measure on $\HE^n$ (with respect the metric $d$).
By \cite{MR1096400} there exist constants $\alpha, D \in ]1,\infty[$ (depending only on $k$ and the regularity constant $C_E$)
and a collection $\Delta^* = \bigcup_{j\in\Z} \Delta_j \subset \mathcal{P}(E)$ such that each $Q\in\Delta^*$ is open in $E$ and
\begin{align} 
\label{D1}  
  &\text{$\mu\biggl(E\backslash \bigcup_{Q\in\Delta_j} Q\biggr) = 0$ for all $j\in\Z$.}  \\
\label{D2}  
  &\text{If $Q, R \in \Delta_j$ and $Q \neq R$, then $Q \cap R = \emptyset$.}  \\
\label{D3}  
  &\text{If $Q\in\Delta_j$, $R\in\Delta_l$ and $j \leq l$, then $Q \subset R$ or $Q \cap R = \emptyset$.}  \\
\label{D4}  
  &\text{$d(Q) \leq D\alpha^j$ for all $Q\in\Delta_j$.}  \\
\label{D5}  
  &\text{If $Q \in \Delta_j$, then $B(x,D^{-1}\alpha^j) \cap E \subset Q$ for some $x\in E$.}  \\
\label{D6}  
  &\text{$\mu\left(\{ x\in Q\, :\, d(x,E\backslash Q) \leq t\alpha^j \}\right) \leq Dt^{1/D}\mu(Q)$ for all $Q\in\Delta_j$, $t > 0$.}
\end{align} 
By \eqref{D4} and \eqref{D5} also $D^{-k}C_E^{-3}\alpha^{jk} \leq \mu(Q) \leq C_E D^k\alpha^{jk}$  for $Q\in\Delta_j$
if $\alpha^j \leq Dd(E)$. Thus by defining $J_0 = \inf\{ j \, :\, \Delta_j = \{E\} \}$ (here $J_0 = \infty$ if $d(E) = \infty$) and
$\ZZ = \{ j\in\Z\, :\, j\leq J_0 \}$
and taking $D$ larger we can assume that
\begin{align}
\label{D}
  D^{-1}\alpha^j \leq d(Q) \leq D\alpha^j \quad\text{and}\quad D^{-1}\alpha^{jk} \leq \mu(Q) \leq D\alpha^{jk}
\end{align}
for all $Q\in\Delta_j$, $j\in\ZZ$. Set $\Delta = \bigcup_{j\in\ZZ} \Delta_j$.

If $(X,\rho) \in \{(\HE^n,d),(\R^k,d_e) \}$ we write $Z(r) = \{  x\in X\, :\, \rho(x,Z) \leq r\}$ for any $Z\subset X$ and $r>0$.
We further denote
\begin{align*}
  \lambda Q &= Q((\lambda-1)d(Q)) \cap E, \\
  \lambda F &= Q((\lambda-1)d_e(F))
\end{align*}
for any $Q\subset E$, $F\subset\R^k$ and $\lambda > 1$.
Each constant in this article may depend on $k$ and $C_E$ without special mention.
For future let $\varepsilon$ and $\delta$ be small positive constants and $K_0$ and $K$ large constants.
We will fix $K_0$ first, $\delta$ second and $\varepsilon$ after $K$. 
The constants $C$ in Sections~\ref{secbd}--\ref{secF3} depend on
$\varepsilon$, $K$, $\delta$ or $K_0$ only if it is separately mentioned.
Eventually every constant will depend only on $k$, $C_E$ and $C$ from \eqref{oletus}.

For $x\in \HE^n$, $t>0$ and $F \subset E$ with $d(F)>0$ denote
\begin{align}
\label{beta1}
  \beta_1(x,t) = \beta^E_1(x,t) &=  t^{-k-1}\inf_{V\in\V^k} \int_{B(x,t)} d(y,V)\, d\mu y, \\
\notag 
  \beta_\infty(F) &= d(F)^{-1}\inf_{V\in\V^k}\sup\left\{\, d(y,V)\, :\, y\in F\, \right\}.
\end{align}
We say that $E$ satisfies the \emph{weak geometric lemma} if for each $\lambda_1>0$ and $\lambda_2 > 1$
there is a constant $C(\lambda_1,\lambda_2)$ such that
\begin{align}
\label{wgl}
  \sum_{j\in\ZZ}\sum_{\substack{Q\in\Delta_j \\ Q\subset R \\ \beta_\infty(\lambda_2 Q) > \lambda_1}} \mu(Q) \leq C(\lambda_1,\lambda_2)\mu(R)
  \qquad\text{for all $R\in\Delta$}.
\end{align}
Denote
\begin{align*}
  \G_1 = \left\{\, Q\in\Delta\, :\, \text{$KQ \subset V(\varepsilon^2 d(Q))$ for some $V\in\V^k$}\, \right\}.
\end{align*}
Clearly \eqref{wgl} implies
\begin{align}
\label{C0}
  \sum_{j\in\ZZ}\sum_{\substack{Q\in\Delta_j\backslash\G_1 \\ Q\subset R}} \mu(Q) \leq C(\varepsilon,K)\mu(R)
  \qquad\text{for all $R\in\Delta$}.
\end{align}
Note also that \eqref{oletus} implies \eqref{wgl} (and hence \eqref{C0}).
For the proof see for example \cite{MR1113517}.
(This clearly remains valid even if $\HE^n$ and $\V^k$ are replaced by any $k$-regular metric space $(X,\rho)$
and $\A\subset \mathcal{P}(X)$ with $\inf\{ \rho(x,V)\, :\, V\in\A \} = 0$ for all $x\in E$.)
For each $Q\in\G_1$ we let $V_Q\in\V^k$ be such that $KQ \subset V_Q(\varepsilon^2 d(Q))$.

\begin{lemma}
\label{leL}
  There is a constant $c>0$ such that for any $Q\in\G_1$ there exists
  $\{y_0,\dotsc,y_k\} \subset V_Q \cap Q(\varepsilon^2 d(Q))$ such that
  $d(y_{i+1},L_i) > cd(Q)$ for all $i\in\{0,\dotsc,k-1\}$,
  where $L_i \subset V_Q$ is the $i$-dimensional affine subspace with $\{y_0,\dotsc,y_i\} \subset L_i$.
\end{lemma}

\begin{proof}
Choose some $x_0 \in Q$ and take $L_0=\{y_0\} \subset V_Q$ such that $d(x_0,y_0) = d(x_0,V_Q) \leq \varepsilon^2 d(Q)$.
Assume now that $i<k$ and the $i$-dimensional affine subspace $L_i \subset V_Q$ is defined.
Suppose to the contrary that
$V_Q \cap Q(\varepsilon^2 d(Q)) \subset L_i(c d(Q))$, where $c>0$ is a constant determined later.
Then $Q \subset L_i((c+\varepsilon)d(Q))$.
Denote $F = L_i \cap Q((c+\varepsilon)d(Q))$ and let $H$ be a maximal subset of $F$ such that $d(z,w) > ad(Q)$ for distinct $z,w\in H$,
where $a>0$ is a constant fixed later.
Then
\[ Q \subset F((c+\varepsilon)d(Q)) \subset \bigcup_{y\in H} B(y,(c+\varepsilon+a)d(Q)) \]
and
\[ \#H\cdot \left(\frac{a d(Q)}{2}\right)^i \leq (1+2(c+\varepsilon))^id(Q)^i \]
by \eqref{dV}.
Picking $x(y)\in B(y,(c+\varepsilon+a)d(Q))$ for each $y\in H$ we get
\begin{align*}
  \mu(Q) &\leq \sum_{y\in H} \mu(B(x(y),2(c+\varepsilon+a)d(Q)) \leq \#H\cdot C_E 2^ k(c+\varepsilon+a)^kd(Q)^k  \\
  &\leq \left(\frac{2+4(c+\varepsilon)}{a}\right)^i C_E 2^k(c+\varepsilon+a)^kd(Q)^k.
\end{align*}
By choosing $a$ and $c$ suitably (depending only on $k$ and $C_E$) and then $\varepsilon>0$ small enough one gets the contradiction
with \eqref{D}.
\end{proof}

\begin{lemma}
\label{leQP}
  If $Q\in\G_1$ and $V\in\V^k$ is such that $Q \subset V(2K\varepsilon^2 d(Q))$,
  then $\kulma{V_Q}{V} \leq 1 + \varepsilon$.
\end{lemma}

\begin{proof}
Let $\{y_0,\dotsc,y_k\} \subset V_Q \cap Q(\varepsilon^2 d(Q))$ as in Lemma~\ref{leL}.
Then $\{y_0,\dotsc,y_k\} \subset V((1+2K)\varepsilon^2 d(Q))$ and $d(y_i,y_j) > cd(Q)$ for all distinct $i,j\in\{0,\dotsc,k\}$.
Thus by Lemma~\ref{lePdist}
\begin{align*}
  d(y_i,y_j) &\leq d(P_V(y_i),P_V(y_j)) + d(y_i,P_V(y_i)) + d(y_j,P_V(y_j))  \\
  &\leq d(P_V(y_i),P_V(y_j)) + 6(1+2K)\varepsilon^2 d(Q)  \\
  &\leq d(P_V(y_i),P_V(y_j)) + 6(1+2K)c^{-1}\varepsilon^2 d(y_i,y_j)
\end{align*}
for all $i,j\in\{1,\dotsc,k\}$, where $c$ is as in Lemma~\ref{leL}.
Choosing $\varepsilon$ small enough depending on $K$ (and $c$) we get by \eqref{PV} and \eqref{dV} that
$|x-y| \leq (1+\varepsilon)|P^e_{V'}(x)-P^e_{V'}(y)|$ for all $x,y\in V_Q'$.
The claim now follows from Lemma~\ref{lekulmat}.
\end{proof}

\section{Stopping time regions}
\label{secF}

In this section we mostly follow \cite[Sections 7 and 8]{MR1113517}. We use the same assumptions and notations as in Section~\ref{secbd}
assuming additionally that $E$ satisfies the weak geometric lemma~\eqref{wgl}.
For any $Q\in\Delta^*$ denote $\CH(Q) = \{ R\in\Delta_{j_Q-1}\, :\, R \subset Q \}$, where $j_Q=\min\{\, j\in\Z\, :\, Q\in\Delta_j\, \}$.
If $j_Q < J_0$, we also denote by $O(Q)$ the unique $R\in\Delta^*$ for which $Q\in\CH(R)$.
For any $\ST\subset\Delta$ let $\min(\ST)$ be the set of minimal (with respect to inclusion) cubes in $\ST$.

\begin{lemma}
There is $\G\subset\G_1$ and $\F\subset\mathcal{P}(\G)$ such that $\G = \bigcup_{\ST\in\F}\ST$ and
the following conditions are satisfied:
\begin{itemize}
\item[\rm(F1)] For all $R\in\Delta$
  \[ \sum\limits_{j\in\ZZ}\sum\limits_{\substack{Q\in\Delta_j\backslash\G \\ Q\subset R}} \mu(Q) \leq C(\varepsilon,K)\mu(R). \]
\item[\rm(F2)]
  If $\ST_1,\ST_2\in\F$ and $\ST_1\neq\ST_2$, then $\ST_1\cap\ST_2 = \emptyset$.
\item[\rm(F3)]
  Each $\ST\in\F$ has a largest element with respect to inclusion, denoted by $Q(\ST)$.
\item[\rm(F4)]
  If $Q \in \ST$, $R\in\Delta$ and $Q\subset R \subset Q(\ST)$, then $R \in \ST$.
\item[\rm(F5)]
  $\kulma{V_Q}{V_{Q(\ST)}} \leq 1+\delta$ for all $Q\in\ST$.
\item[\rm(F6)]
  If $Q\in\ST$, $\CH(Q)\subset\G$ and $\kulma{V_R}{V_{Q(\ST)}} \leq 1+\delta$ for all $R\in\CH(Q)$, then $\CH(Q)\subset\ST$.
\item[\rm(F7)]
  $Q\in\min(\ST)$ if and only if the following two conditions are satisfied:
  \begin{itemize}
  \item[\textbullet] $Q\in\ST$
  \item[\textbullet] $\CH(Q)\backslash\G\neq\emptyset$ or $\kulma{V_R}{V_{Q(\ST)}} > 1+\delta$ for some $R\in\CH(Q)$
  \end{itemize}
\end{itemize}
\end{lemma}

\begin{proof}
Assume first that $E$ is unbounded.
Let $p\in E$ and set $\mathcal{D} = \min\bigl(\bigcup_{j\in\N}\mathcal{D}_j\bigr)$, where
$\mathcal{D}_j = \left\{\, R\in\Delta_j\, :\, B(p,\alpha^j) \cap O(R)\neq\emptyset\, \right\}$.
Clearly each $Q\in\Delta$ is included in some $R\in\bigcup_{j\in\N}\mathcal{D}_j$.
Let $j\in\Z$ and $Q\in\Delta_j$ be such that $d(p,Q) > (D^3\alpha+1)\alpha^{j}$.
We next show that there exists $R\in\mathcal{D}$ such that $Q\subset R$.
We just let $R$ be the minimal qube in $\bigcup_{i\in\N}\mathcal{D}_i$ such that $Q\subset R$.
Since trivially $\mathcal{D}_0 \subset \mathcal{D}$, we assume $R\in\mathcal{D}_i$ for $i\geq 1$.
Now $i>j$ because $Q\not\in\mathcal{D}_j$. Thus $Q\subset R^*$ for some $R^* \in \CH(R)$.
By the minimality of $R$ we have $B(p,\alpha^{j_R-1}) \cap R = \emptyset$ from which we conclude $R\in\mathcal{D}$.
So by the regularity there is a constant $C$ such that for every $j\in\Z$ there is at most $C$ cubes in $\Delta_j$
which are not contained in
any cube in $\mathcal{D}$. 
If $E$ is bounded we set $\mathcal{D} = \Delta_{J_0}$ (which contains only one qube).

Defining $\G = \G_1 \backslash \{\, Q \in \Delta\, :\, \text{$Q\not\subset R$ for all $R\in\mathcal{D}$}\, \}$ the condition~(F1) holds by
\eqref{C0} and the previous discussion.
For each $R\in\mathcal{D}$ we partition $\G(R) = \{ Q\in\G\, :\, Q \subset R \}$ into a family of
''stopping time regions'' as follows:
Let $Q_0$ be a maximal element in $\G(R)$. The family $\ST$ is defined to be the unique subset of $\G(R)$
whose largest element is $Q_0$ and which satisfies the conditions~(F3)--(F7).
Then we repeat the process for $\G(R)\backslash\ST$.
Since $R_1 \cap R_2 = \emptyset$ for distinct $R_1,R_2\in\mathcal{D}$, the condition~(F2) is satisfied.
\end{proof}

Notice that (F6) and (F7) imply
\begin{align}
\label{C6b}
  Q \in \ST\backslash\min(\ST) \Longrightarrow \CH(Q)\subset\ST.
\end{align}
For $\ST\in\F$ denote
\begin{align*}
  m_1(\ST) &= \{\, Q\in\min(\ST)\, :\, \CH(Q)\backslash\G\neq\emptyset\, \},  \\
  m_2(\ST) &= \min(\ST) \backslash m_1(\ST)
\end{align*}
and further
\begin{align*}
  \F_1 &= \Bigl\{\, \ST\in\F\, : \mu\Bigl(\bigcup_{Q\in m_1(\ST)} Q\Bigr) \geq \mu(Q(\ST))/4\, \Bigr\},  \\
  \F_2 &= \Bigl\{\, \ST\in\F\, : \mu\Bigl( Q(\ST) \backslash\bigcup_{Q\in\min(\ST)} Q\Bigr) \geq \mu(Q(\ST))/4\, \Bigr\},  \\
  \F_3 &= \Bigl\{\, \ST\in\F\, : \mu\Bigl(\bigcup_{Q\in m_2(\ST)} Q\Bigr) \geq \mu(Q(\ST))/2\, \Bigr\}.
\end{align*}
Clearly $\F = \F_1\cup\F_2\cup\F_3$.

\begin{lemma}
\label{leF1F2}
  There is a constant $C=C(\varepsilon,K)$ such that
  \[ \sum_{\substack{\ST\in\F_1 \cup\F_2 \\ Q(\ST)\subset R}} \mu(Q(\ST)) \leq C\mu(R)\qquad\text{for all $R\in\Delta$}. \]
\end{lemma}

\begin{proof}
Let $R\in\Delta$. By \eqref{D} and (F1)
\begin{align*}
  \sum_{\substack{\ST\in\F_1 \\ Q(\ST)\subset R}} \mu(Q(\ST))
  \leq 4\sum_{\substack{\ST\in\F_1 \\ Q(\ST)\subset R}} \mu\Bigl(\bigcup_{Q\in m_1(\ST)} Q\Bigr)
  \leq 4D^2\alpha^k\sum_{\substack{Q\in\Delta\backslash\G \\ Q\subset R}} \mu(Q) \leq 4D^2\alpha^kC(\varepsilon,K)\mu(R).
\end{align*}
Let $\ST_1,\ST_2\in\F$, $\ST_1\neq\ST_2$. Then $Q(\ST_1)\neq Q(\ST_2)$ by (F2). If $Q(\ST_1) \cap Q(\ST_2)\neq\emptyset$ then \eqref{D3} implies
$Q(\ST_1) \subset Q(\ST_2)$ or $Q(\ST_2) \subset Q(\ST_1)$. Assume that $Q(\ST_1) \subset Q(\ST_2)$ and take minimal $Q\in\ST_2$ such that
$Q(\ST_1) \subset Q$. Since $Q \neq Q(\ST_1)$ by (F2) one has $Q\in\min(\ST_2)$ by \eqref{C6b}.
Thus the sets $Q(\ST) \backslash\bigcup_{Q\in\min(\ST)} Q$, $\ST\in\F$, are disjoint and
\begin{align*}
  \sum_{\substack{\ST\in\F_2 \\ Q(\ST)\subset R}} \mu(Q(\ST))
  \leq 4\sum_{\substack{\ST\in\F_2 \\ Q(\ST)\subset R}} \mu\Bigl( Q(\ST) \backslash\bigcup_{Q\in\min(\ST)} Q\Bigr)
  \leq 4\mu(R).
\end{align*}
\end{proof}

Now the goal is to show that
\begin{align}
\label{t}
  \sum_{\substack{\ST\in\F_3 \\ Q(\ST)\subset R}} \mu(Q(\ST)) \leq C\mu(R)\qquad\text{for all $R\in\Delta$}.
\end{align}
The full assumption~\eqref{oletus} will be used (instead of the weaker condition~\eqref{wgl}) only on page~\pageref{ol} to get \eqref{t}.
After this Theorem~\ref{th} follows quite easily by the following lemma (see Section~\ref{end}).

For any $\ST\in\F$ define the function $h_\ST:\HE^n\to\R$ by setting
\begin{align*}
  h_\ST(x) = \inf\{\, d(x,Q) + d(Q)\, :\, Q\in\ST\, \}.
\end{align*}

\begin{lemma}
\label{ledh}
If $\ST\in\F$ and $x,y\in K_0Q(\ST)$ with $d(x,y) > D^{-2}\min\{h_\ST(x),h_\ST(y)\}$, then
$d(x,y) \leq (1+2\delta)d(P_{V_{Q(\ST)}}(x),P_{V_{Q(\ST)}}(y))$.
\end{lemma}

\begin{proof}
Assume that $d(x,y) > D^{-2}h_\ST(x)$ and choose $Q\in\ST$ such that
\[ d(x,y) > D^{-2}(d(x,Q)+d(Q)). \]
Let $R\in\ST$ be the minimal cube such that $Q\subset R$ and $d(K_0R)\geq d(x,y)$.
Then
\[ d(y,R) \leq d(x,y) + d(x,R) \leq d(x,y) + d(x,Q) < (1+D^2)d(x,y) \]
and $d(R) \leq D^2(1+\alpha)d(x,y)$ by \eqref{D}.
Let $z,w\in V_R$ with $d(x,z) = d(x,V_R)$ and $d(y,w) = d(y,V_R)$.
Choosing $K$ large enough (depending on $K_0$ and $D$) and denoting $P=P_{V_{Q(\ST)}}$ one gets by (F5)
\begin{align*}
  &d(P(x),P(y)) \geq d(P(z),P(w)) - d(P(x),P(z)) - d(P(y),P(w))  \\
  &\geq d(P(z),P(w)) - d(x,z) - d(y,w) \geq (1+\delta)^{-1}d(z,w) - 2\varepsilon d(R)  \\
  &\geq (1+\delta)^{-1}(d(x,y) - 2\varepsilon d(R)) - 2\varepsilon d(R)
  > ((1+\delta)^{-1} - 4D^2\varepsilon(1+\alpha))d(x,y).
\end{align*}
The claim now follows by choosing $\varepsilon$ small enough (depending on $\delta$).
\end{proof}

\section{Function $g$ for $\ST$}
\label{secS}

In this section we follow \cite[Section 8]{MR1113517} and use the same assumptions and notations as in Section~\ref{secF}.
Let $\ST\in\F$ be fixed and assume (in order to simplify notations) that $V_{Q(\ST)}=X_k$.
Then $P_{V_{Q(\ST)}}=P^e_{X_k}$ and $d(p,q)=|p-q|$ for any $p,q\in X_k$.
Because of the latter fact it is natural to denote by $d(F)$ the euclidean diameter of $F$ and by $d(p,F)$ the euclidean distance of $p$ and $F$
for any $F\subset\R^k$ and $p\in\R^k$ (though $d$ is a metric in $\HE^n$).
We write $P(x)=(x_1,\dotsc,x_k)$ and $P^{\bot}(x)=(x_{k+1},\dotsc,x_{2n})$ for any $x\in\HE^n$.
Denote also $B^k(p,r) = \{ q\in\R^k\, :\, |q-p| \leq r \}$ for $p\in\R^k$ and $r\geq 0$.
The letter $C$ in the calculations in Sections~\ref{secS}--\ref{end} denotes always some constant
but distinct appearances do not necessarily refer to the same constant
(even if they are in the same inequality chain).

Define the function $H:\R^k\to\R$ by setting
\begin{align*}
  H(p) = \inf\{\, h(x) \, :\, P(x)=p\, \}
\end{align*}
and set $Z = \{ x\in E\, :\, h(x)=0 \}$.
Here we write shortly $h=h_\ST$.
We immediately see that
\begin{align}
\label{H}
  H(p) = \inf\{\, d(p,P(Q))+d(Q) \, :\, Q\in\ST\, \}.
\end{align}
for any $p\in\R^k$. Namely, the inequality $H(p) \geq \inf\{\, d(p,P(Q))+d(Q) \, :\, Q\in\ST\, \}$ follows from the 1-Lipschitzness of $P$.
The opposite inequality holds because for any $y\in\HE^n$ and $p\in\R^k$ one can obviously choose $x\in P^{-1}(\{p\})$ such that
$d(x,y)=d(p,P(y))$. Note that $H(p)=0$ if and only if $p\in P(Z)$ (for example by the Bolzano--Weierstrass theorem).

For each $p\in\R^k\backslash P(Z)$ let $R_p$ be the largest dyadic cube in $\R^k$
containing $p$ and satisfying
\begin{align}
\label{defR}
  20d(R_p) \leq \inf\{\, H(u)\, : u \in R_p\, \}.
\end{align}
Such a cube $R_p$ exists, because $H(p)>0$ and $H$ is continuous (1-Lipschitz).
Let $\{ R_i\, :\, i\in I\} \subset \{ R_p\, : p\in \R^k\backslash P(Z)\}$ be such that $\{R_i\}_{i\in I}$ covers $\R^k\backslash P(Z)$ and
$\sisus R_i\cap\sisus R_j = \emptyset$ for distinct $i,j\in I$. Notice that $I$ is countable and $R_i\cap P(Z)=\emptyset$ for any $i\in I$.

By the definition~\eqref{defR} and the 1-Lipschitzness of $H$
\begin{align}
\label{RH}
  10d(R_i) \leq H(p) \leq 60d(R_i)\qquad\text{for any $p\in 10R_i$, $i\in I$.}
\end{align}
This gives the following lemma.

\begin{lemma}
\label{leR}
There is a constant $C$ such that whenever $10R_i\cap 10R_j\neq\emptyset$ for $i,j\in I$ then
$C^{-1}d(R_j) \leq d(R_i) \leq Cd(R_j)$.
\end{lemma}

Let $x_0\in Q(\ST)$ be any fixed point. Denote
$U_j = B^k(P(x_0),2^{-j}K_0d(Q(\ST)))$ and
$I_j = \{ i\in I\, :\, R_i\cap U_j\neq\emptyset\}$
for $j\in\R$.
By \eqref{H} and \eqref{RH} there exist constants $C_0$ (which may depend on $K_0$ by $C_0=CK_0$) and $C$ such that
for each $i\in I_0$ there is $Q_i\in\ST$ for which
\begin{align}
\label{Qi1}
  &C_0^{-1}d(R_i) \leq d(Q_i) \leq Cd(R_i), \\
\label{Qi2}
  &d(P(Q_i) \cup R_i) \leq Cd(R_i).
\end{align}
(In \eqref{Qi2} we use the fact that $P$ is 1-Lipschitz.)

For each $i\in I_0$ let $A_i : \R^k \to \R^{2n-k}$ be the affine function whose graph is $V_{Q_i}'$.
By Lemma~\ref{lekulmat} and (F5) (and by choosing $\delta\leq 1$)
\begin{align}
\label{LipB}
  \Lip(A_i) \leq \sqrt{(1+\delta)^2-1} < 2\sqrt{\delta}.
\end{align}

\begin{lemma}
\label{leA}
There is a constant $C$ such that whenever $10R_i\cap 10R_j\neq\emptyset$ for $i,j\in I_0$ then
$d(Q_i \cup Q_j) \leq Cd(R_j)$ and
\[ |A_i(p)-A_j(p)| \leq C\sqrt{\varepsilon}d(R_j)\qquad\text{for all $p\in 100R_j$}. \]
\end{lemma}

\begin{proof}
For the first part let $x,y\in Q_i \cup Q_j$. By \eqref{Qi1} and Lemma~\ref{leR} one may assume that $x\in Q_i$ and $y\in Q_j$ are such that $d(x,y)\geq d(Q_j)$.
Since by definition $h(y)\leq d(Q_j)$, Lemma~\ref{ledh}, \eqref{Qi2} and Lemma~\ref{leR} give
\[ d(x,y) \leq (1+2\delta)|P(x)-P(y)| \leq Cd(R_j). \]
Thus by choosing $K$ large enough depending on $K_0$ \eqref{Qi1} gives $Q_i\subset KQ_j$ (and $Q_j\subset KQ_i$).
Now Lemma~\ref{leQP} (with $Q=Q_j$ and $V=V_{Q_i}$) gives
\begin{align*}
 \kulma{V_{Q_j}}{V_{Q_i}} \leq 1+\varepsilon.
\end{align*}
Let $p\in 100R_j$ and $z\in Q_j$. Take $y\in V_{Q_j}$ such that $d(y,z)\leq\varepsilon d(Q_j)$. Then by \eqref{Qi1} and Lemma~\ref{leR}
\[ |y'-P^e_{V_{Q_i}'}(y')| = d_e(y',V_{Q_i}') \leq d(y,z)+d(z,V_{Q_i}) \leq \varepsilon(d(Q_i)+d(Q_j)) \leq C\varepsilon d(R_j) \]
and by \eqref{Qi1} and \eqref{Qi2}
\[ |P(y)-p| \leq |P(y)-P(z)|+|P(z)-p| \leq d(y,z)+ d(P(Q_j)\cup 100R_j) \leq Cd(R_j). \]
Denote $v=(p,A_j(p))$. Using the Pythagorean theorem, the above estimates, Lemma~\ref{lekulmat} and (F5)
\begin{align*}
  |v-P^e_{V_{Q_i}'}(v)|^2 &= |v-P^e_{V_{Q_i}'}(y')|^2 - |P^e_{V_{Q_i}'}(v)-P^e_{V_{Q_i}'}(y')|^2 \\
  &\leq (|v-y'|+|y'-P^e_{V_{Q_i}'}(y')|)^2 - |P^e_{V_{Q_i}'}(v)-P^e_{V_{Q_i}'}(y')|^2 \\
  &\leq (|v-y'|+C\varepsilon d(R_j))^2 - (1+\varepsilon)^{-2}|v-y'|^2 \\
  &\leq C\varepsilon d(R_j)^2
\end{align*}
and therefore by (F5)
\[ |A_i(p)-A_j(p)| \leq |v-P^e_{V_{Q_i}'}(v)|+(1+\delta)|P(P^e_{V_{Q_i}'}(v))-P(v)| \leq C\sqrt{\varepsilon}d(R_j). \]
\end{proof}

For each $i\in I_0$ let $\tilde{\phi}_i:\R^k\to [0,1]$ be  a $C^2$ function such that
$\tilde{\phi}_i(p)=1$ for all $p\in 2R_i$, $\tilde{\phi}_i(p)=0$ for all $p\in\R^k\backslash 3R_i$ and
\begin{align}
\label{phitilde}
\begin{split}
  |\partial_j\tilde{\phi}_i| &\leq Cd(R_i)^{-1}, \\
  |\partial_j\partial_m\tilde{\phi}_i| &\leq Cd(R_i)^{-2}
\end{split}  
\end{align}
for all $j,m\in\{1,\dotsc,k\}$.
Set
\begin{align*}
  \phi_i(p) = \frac{\tilde{\phi}_i(p)}{\sum_{j\in I_0}\tilde{\phi}_j(p)}\qquad\text{for any $p\in U_0\backslash P(Z)$, $i\in I_0$}.
\end{align*}
For each $p\in P(Z)$ there is $x(p)$ such that $P^{-1}(\{p\}) \cap K_0Q(\ST)= \{x(p)\}$ by Lemma~\ref{ledh}.
We now define a function $g:U_0\to\R^{2n-k}$ by setting
\begin{align*}
  g(p) = \begin{cases}
    \sum\limits_{i\in I_0} \phi_i(p)A_i(p), &\text{if $p\in U_0\backslash P(Z)$}  \\
    P^{\bot}(x(p)), &\text{if $p\in P(Z)$}.
  \end{cases}
\end{align*}

\begin{lemma}
\label{leglip}
The function $g$ is $C\sqrt{\delta}$-Lipschitz.
\end{lemma}

\begin{proof}
By taking $\varepsilon/\delta$ small enough we get as in \cite[equation (8.19)]{MR1113517} that
\begin{align}
\label{g2Rj}
  |g(p)-g(q)| \leq 3\sqrt{\delta}|p-q|\qquad\text{for $p,q\in 2R_j \cap U_0$, $j\in I_0$}.
\end{align}
By Lemma~\ref{ledh}
\begin{align}
\label{gPZ}
  |P^{\bot}(y)-g(q)| \leq 2\sqrt{\delta(1+\delta)}|P(y)-q|\qquad\text{for $q\in P(Z)$, $y\in Q(\ST)$}.
\end{align}
Let for a while $j\in I_0$, $p\in R_j \cap U_0$, $y\in Q_j$ and $q\in P(Z)$.
Then
\begin{align*}
  |g(p)-A_j(p)| \leq C\sqrt{\varepsilon}d(R_j)
\end{align*}
by the definition of $g$ and Lemma~\ref{leA}
(because the supports of the functions $\phi_i$ have bounded overlap by Lemma~\ref{leR}), and
\begin{align*}
  |A_j(p)-A_j(P(y))| \leq 2\sqrt{\delta}|p-P(y)| \leq C\sqrt{\delta}d(R_j)
\end{align*}
by \eqref{LipB} and \eqref{Qi2}. By (F5)
\begin{align*}
  |A_j(P(y))-P^{\bot}(y)| \leq (2+\delta)\varepsilon d(Q_j).
\end{align*}
Here $H(p) = H(p)-H(q) \leq |p-q|$. Thus by \eqref{gPZ}, \eqref{defR}, \eqref{Qi1} and \eqref{Qi2}
(choosing $\varepsilon$ small enough depending on $\delta$ and $K$)
\begin{align}
\label{gPZR}
  |g(p)-g(q)| \leq C\sqrt{\delta}|p-q|\qquad\text{for $p\in P(Z)$, $q\in U_0\backslash P(Z)$}.
\end{align}
Lemma now follows easily from \eqref{g2Rj}, \eqref{gPZ} and \eqref{gPZR}.
\end{proof}

\begin{lemma}
\label{leC1}
There is a constant $C=C(K_0)$ such that
$P^{-1}(\{p\}) \cap K_0Q(\ST) \subset CQ_i$
for all $p\in R_i$, $i\in I_0$.
\end{lemma}

\begin{proof}
Let $i\in I_0$, $p\in R_i$, $x\in P^{-1}(\{p\}) \cap K_0Q(\ST)$ and $y\in Q_i$. We may assume that $d(x,y) > d(Q_i)$ (since otherwise $x\in 2Q_i$).
Then $d(x,y) > h(y)$ and Lemma~\ref{ledh} (by choosing $\delta$ small), \eqref{Qi2} and \eqref{Qi1} yield
$d(x,y) \leq 2|p-P(y)| \leq Cd(Q_i)$.
\end{proof}

\begin{lemma}
\label{leC2}
There is a constant $C=C(K_0)$ such that
$h(x) \leq C H(P(x))$ for all $x\in K_0Q(\ST)$.
\end{lemma}

\begin{proof}
Let $x\in K_0Q(\ST)$. By Lemma~\ref{ledh} one may assume that $P(x)\not\in P(Z)$.
Let $i\in I_0$ such that $P(x)\in R_i$ (such $i$ exists because $P$ is 1-Lipschitz).
Then 
$h(x) \leq d(x,Q_i)+d(Q_i) \leq Cd(Q_i) \leq CH(P(x))$
by the previous lemma, \eqref{Qi1} and \eqref{RH}.
\end{proof}

\begin{lemma}
\label{legappr}
There exists a constant $C$ such that
\[ |P^{\bot}(x)-g(P(x))| \leq C\sqrt{\varepsilon}h(x) \]
for all $x\in K_0Q(\ST)$.
\end{lemma}

\begin{proof}
Let $x\in K_0Q(\ST)$ with $h(x)>0$. Now $H(P(x))>0$ by the previous lemma and so $P(x)\in R_i$ for some $i\in I_0$
and further $x\in CQ_i$ by Lemma~\ref{leC1}. As in the proof of Lemma~\ref{leglip}, choosing $K$ large enough depending on $K_0$
we get
$|A_i(P(x))-P^{\bot}(x)| \leq (2+\delta)\varepsilon d(Q_i)$
by (F5).
Since also
$|g(P(x))-A_i(P(x))| \leq C\sqrt{\varepsilon}d(R_i)$
by the definition of $g$ and Lemma~\ref{leA},
we get the result by \eqref{Qi1} and \eqref{RH}.
\end{proof}

\begin{lemma}
\label{ledA}
There is a constant $C$ such that whenever $10R_i\cap 10R_j\neq\emptyset$ for $i,j\in I_0$ then
\[ |\partial_m(A_i(p)-A_j(p))| \leq C\sqrt{\varepsilon}\qquad\text{for all $m\in\{1,\dotsc,k\}$ and $p\in\R^k$}. \]
\end{lemma}

\begin{proof}
Since $\partial_m A_i$ is constant for any $m$ and $i$ it is enough to prove the claim for a fixed
$p\in 10R_i\cap 10R_j$. Let $t=d(R_j)$. Then by Lemma~\ref{leA}
\[ |A_i(p+te_m)-A_j(p+te_m)-(A_i(p)-A_j(p))| \leq C\sqrt{\varepsilon}t, \]
which gives the result, because for any $i$ the quotient $t^{-1}(A_i(p+te_m)-A_i(p))$ does not depend on $t$.
\end{proof}

Using \eqref{phitilde} and Lemmas~\ref{leR}, \ref{leA} and \ref{ledA} one gets (see \cite[Lemma 8.22]{MR1113517})

\begin{lemma}
\label{leddg}
There is a constant $C$ such that
\[ |\partial_j\partial_mg(p)| \leq \frac{C\sqrt{\varepsilon}}{d(R_i)}\qquad\text{for all $j,m\in\{1,\dotsc,k\}$ and $p\in 2R_i\cap\sisus U_0$, $i\in I_0$}. \]
\end{lemma}

\begin{lemma}
\label{legsup}
There is a constant $C$ such that
$|g(p)| \leq CK_0\sqrt{\delta}d(Q(\ST))$ for all $p\in U_0$.
\end{lemma}

\begin{proof}
Let $p\in U_0$, $i\in I_0$, $x\in Q_i$ and $y\in V_{Q_i}$ with $d(x,y)\leq\varepsilon d(Q_i)$.
By Lemma~\ref{lePdist}
\begin{align*}
  |A_i(P(y))| = |P^\bot(y)| \leq |P^\bot(x)| + |P^\bot(y-x)| \leq 3d(x,X_k) + d(x,y) \leq 4\varepsilon d(Q(\ST)).
\end{align*}
Further by \eqref{LipB} (and the Lipschitzness of $P$)
\begin{align*}
  |A_i(p)-A_i(P(y))| &\leq 2\sqrt{\delta}|p-P(y)| \leq 2\sqrt{\delta}(|p-P(x)| + |P(x)-P(y)|) \\
  &\leq 2\sqrt{\delta}((K_0+1)d(Q(\ST)) + \varepsilon d(Q_i)).
\end{align*}
Thus $|A_i(p)| \leq CK_0\sqrt{\delta}d(Q(\ST))$.
The desired estimate for $|g(p)|$ now follows from Lemma~\ref{leR}.
\end{proof}

\section{Function $\gamma$ for $\ST$}
\label{secgamma}

In this section we use the same assumptions and notations as in Section~\ref{secS}.
Using Lemma~\ref{legsup} we extend $g$ from $U_0$ to a $C\sqrt{\delta}$-Lipschitz
function on $\R^k$ supported in $U_{-1}$.
For $p\in\R^k$ and $t>0$ set
\[ \gamma(p,t) = t^{-k-1} \inf_a \int_{B^k(p,t)} |g(u)-a(u)|\, du, \]
where the infimum is taken over all affine functions $a:\R^k\to\R^{2n-k}$.
Choosing $\delta$ small one has by Lemma~\ref{leglip}
\begin{align}
\label{gammaM}
  \gamma(p,t) \leq 2t^{-k-1} \inf_M \int_{B^k(p,t)} d_e((u,g(u)),M)\, du,
\end{align}
where the infimum is taken over all $k$-planes $M\subset \R^{2n}$.
We follow \cite[Section 13]{MR1113517} and proof the next lemma.

\begin{lemma}
\label{leg}
Let $T=K_0d(Q(\ST))/2$. There exists a constant~$C=C(K_0)$ such that
\[ \int_0^T \int_{U_1}\gamma(p,t)^2\, dp\, \frac{dt}{t}\leq C\varepsilon\mu(Q(\ST))
   +C\varepsilon^{-6k}\int_{K_0Q(\ST)}\int_{h(x)/K_0}^T \beta_1(x,K_0t)^2\, \frac{dt}{t}\, d\mu x. \]
\end{lemma}

Notice that $2R_i \subset U_0$ for all $i\in I_1$ by \eqref{H} and \eqref{defR}.
Using Lemma~\ref{leddg} and Taylor's theorem one gets (see \cite[Lemma 13.7]{MR1113517})

\begin{lemma}
\label{legapu}
There is a constant $C$ such that
\[ \sum_{i\in I_1} \int_0^{d(R_i)} \int_{R_i} \gamma(p,t)^2\, dp\, \frac{dt}{t} \leq C\varepsilon \mu(Q(\ST)). \]
\end{lemma}

We now assume that $p\in U_1$ and $H(p)/60 < t \leq T$.
Choose $z(p,t)\in Q(\ST)$ such that $|p-P(z(p,t))| \leq 60t$ (see \eqref{H})
and let $z\in B(z(p,t),t)\cap E$.
Further let $V_{p,t}\in\V^k$ be such that
\begin{align}
\label{Vpt}
  \int_{B(z,K_0t)} d(x,V_{p,t})\, d\mu x \leq 2(K_0t)^{k+1}\beta_1(z,K_0t).
\end{align}
By \eqref{gammaM}
\begin{align}
\label{apu1}
  2^{-1}t^{k+1}\gamma(p,t) \leq \int_{B^k(p,t)} d_e((u,g(u)),V_{p,t}')\, du.
\end{align}

If $u\in B^k(p,t)\cap P(Z)$ then $P^{-1}(\{u\})\cap Q(\ST)=\{x\}$, where $x'=(u,g(u))$ and $h(x)=0$.
Since $z\in K_0Q(\ST)$ (by choosing $K_0\geq 2$),
\[ d(x,z)\leq (1+2\delta)|u-P(z)| \leq (1+2\delta)(|u-p|+|p-P(z(p,t))|+t)\leq Ct \]
by Lemma~\ref{ledh}.
Thus
\begin{align}
\label{apu2}
  \int_{B^k(p,t)\cap P(Z)} d_e((u,g(u)),V_{p,t}')\, du \leq \int_{B(z,Ct)} d_e(x',V_{p,t}')\, d\mu x.
\end{align}

Let $I(p,t) = \{\, i\in I_0\, : R_i\cap B^k(p,t)\neq\emptyset\, \}$.

\begin{lemma}
\label{lesg1}
There is a constant $C=C(K_0)$ such that for any $i\in I(p,t)$
\[ d_e((u,g(u)),V_{p,t}') \leq d_e((u,g(u)),V_{Q_i}') + \sup\{\, d_e(w',V_{p,t}')\, :\, w\in V_{Q_i} \cap Q_i(Cd(Q_i))\, \} \]
for all $u\in R_i \cap U_0$.
\end{lemma}

\begin{proof}
Let $i\in I(p,t)$ and $u\in R_i$. Denote $y=(u,g(u))$. Let $w\in V_{Q_i}$ be such that $|y-w'| = d_e(y,V_{Q_i}')$.
Further let $v\in V_{p,t}$ be such that $|w'-v'| = d_e(w',V_{p,t}')$. Then
\[ d_e(y,V_{p,t}') \leq |y-v'| \leq |y-w'|+|w'-v'| \leq d_e(y,V_{Q_i}') + d_e(w',V_{p,t}'). \]
By the definition of $g$ and Lemma~\ref{leA}
\[ |u-P(w)| \leq |y-w'| \leq |y-(u,A_i(u))| \leq C\sqrt{\varepsilon}d(R_i). \]
Choosing $q\in Q_i$ and $v_q\in V_{Q_i}$ with $d(v_q,q)\leq \varepsilon d(Q_i)$ one has
by (F5), \eqref{Qi2} and \eqref{Qi1}
\begin{align*}
  d(w,q) &\leq d(w,v_q)+d(v_q,q) \leq (1+\delta)|P(w)-P(v_q)|+d(v_q,q) \\
  &\leq (1+\delta)(|P(w)-u| + |u-P(q)| + |P(q)-P(v_q)|)+d(v_q,q) \leq Cd(Q_i).
\end{align*}
\end{proof}

For any $i\in I(p,t)$
\begin{align}
\label{apu3}
  \int_{B^k(p,t)\cap R_i} d_e((u,g(u)),V_{Q_i}')\, du \leq \int_{B^k(p,t)\cap R_i} |g(u)-A_i(u)|\, du \leq C\sqrt{\varepsilon}d(R_i)^{k+1}
\end{align}
by Lemma~\ref{leA} and the definitions of $A_i$ and $g$. Recall that $B^k(p,t) \subset U_0$ (because $p\in U_1$ and $t\leq T$).

\begin{lemma}
\label{lesg2}
For any constant $C'$ there is a constant~$C$ such that for any $i\in I(p,t)$
\[ d_e(w',V_{p,t}') \leq C\varepsilon d(R_i) + C\varepsilon^{-3k}\left(\dashint_{2Q_i} d_e(x',V_{p,t}')^ {1/3}\, d\mu x \right)^3 \]
for all $w\in V_{Q_i} \cap Q_i(C'd(Q_i))$.
\end{lemma}

\begin{proof}
Let $i\in I(p,t)$ and $w\in V_{Q_i} \cap Q_i(C'd(Q_i))$.
If $y_0,\dotsc,y_k\in V_{Q_i} \cap Q_i(\varepsilon d(Q_i))$ are as in Lemma~\ref{leL} with $Q=Q_i$, then obviously (by \eqref{dV})
\[ d_e(w',V_{p,t}') \leq Cd_e(y_{j_0}',V_{p,t}'), \]
where $d_e(y_{j_0}',V_{p,t}') = \max\{ d_e(y_j',V_{p,t}')\, :\, j\in\{0,\dotsc,k \}\}$.
Let $z_0\in Q_i \cap B(y_{j_0},2\varepsilon d(Q_i))$. Then
\[ d_e(y_{j_0}',V_{p,t}') \leq d_e(x',V_{p,t}') + 3\varepsilon d(Q_i) \] 
for all $x\in B := B(z_0,\varepsilon d(Q_i))$, and we have
\begin{align*}
  \mu(B)d_e(w',V_{p,t}')^{1/3}
  &\leq C\int_B \left(d_e(x',V_{p,t}') + 3\varepsilon d(Q_i)\right)^{1/3}\, d\mu x  \\
  &\leq C\mu(B)\left(3\varepsilon d(Q_i)\right)^{1/3} + C\int_B d_e(x',V_{p,t}')^{1/3}\, d\mu x.
\end{align*}
The claim now follows from the regularity, \eqref{D} and \eqref{Qi1}.
\end{proof}

\begin{lemma}
\label{lesg3}
There is a constant~$C$ such that
\[ \sum_{i\in I(p,t)}\LEK(R_i)\left(\dashint_{2Q_i} d_e(x',V_{p,t}')^ {1/3}\, d\mu x \right)^3 \leq C\int_{B(z,Ct)} d_e(x',V_{p,t}')\, d\mu x. \]
(Here $\LEK$ is the Lebesgue measure on $\R^k$.)
\end{lemma}

\begin{proof}
For any $i\in I$ define $N_i:\HE^n\to\R$ by setting
\begin{align*}
  N_i = \sum_{j\in J(i)} \chi_{2Q_j}
\end{align*}
where $J(i) = \{\, j\in I\, :\, \text{$d(R_j) \leq d(R_i)$ and $2Q_i \cap 2Q_j\neq\emptyset$}\, \}$.
Notice that $N_i(x) \geq 1$ for all $x\in 2Q_i$, $i\in I$.
Let $l,m\in I$. If $2Q_l \cap 2Q_m\neq\emptyset$ and $d(R_l)\leq d(R_m)$, then
$d(R_l \cup R_m) \leq d(P(2Q_l) \cup R_l)+d(P(2Q_m) \cup R_m) \leq Cd(R_m)$
by \eqref{Qi1} and \eqref{Qi2}. Hence by \eqref{D}
\begin{align*}
  \int_{2Q_i} N_i(x)\, d\mu x \leq \sum_{j\in J(i)} \mu(2Q_j) \leq C\sum_{j\in J(i)} \LEK(R_j) \leq C\LEK(R_i),
\end{align*}
and further by H\"older's inequality
\begin{align}
\label{apu4}
  \left( \int_{2Q_i} d_e(x',V_{p,t}')^ {1/3}\, d\mu x \right)^3 \leq C\LEK(R_i)^2\int_{2Q_i} d_e(x',V_{p,t}') N_i(x)^{-2}\, d\mu x
\end{align}
for any $i\in I$.
If $x\in 2Q_l \cap 2Q_m$ and $N_l(x)=N_m(x)$ for $l,m\in I$, then by the definition necessarily $d(R_l)=d(R_m)$ and further (see above)
\begin{align}
\label{apu5}
  \sum_{i\in I} \chi_{2Q_i}(x)N_i(x)^{-2} = \sum_{m=1}^\infty\biggl(m^{-2}\sum_{i\in J(x,m)} \chi_{2Q_i}(x) \biggr) \leq C,
\end{align}
where $J(x,m)=\{ i\in I\, :\, N_i(x)=m \}$.
By \eqref{apu4}, \eqref{apu5} and \eqref{D}
\begin{align}
\label{apu6}
\begin{split}
  &\sum_{i\in I(p,t)}\LEK(R_i)\left(\dashint_{2Q_i} d_e(x',V_{p,t}')^ {1/3}\, d\mu x \right)^3  \\
  &\leq C\sum_{i\in I(p,t)}\int_{2Q_i} d_e(x',V_{p,t}')N_i(x)^{-2}\, d\mu x
  \leq C\int_{\bigcup_{i\in I(p,t)}2Q_i} d_e(x',V_{p,t}')\, d\mu x.
\end{split}
\end{align}
Let $i\in I(p,t)$ and $x\in 2Q_i$.
Since $H(u) \leq H(p)+t \leq 61t$ for all $u\in B(p,t)$, one has $d(R_i)\leq 4t$ by \eqref{defR}.
Thus by \eqref{Qi1} and \eqref{Qi2}
\begin{align*}
 |P(x)-P(z)| \leq d(P(2Q_i)\cup R_i) + d(R_i\cup\{p\}) + |p-P(z(p,t))| + t \leq Ct.
\end{align*}
Since $h(x) \leq 2d(Q_i) \leq Ct$ (by the definition of $h$ and \eqref{Qi1}), Lemma~\ref{ledh} implies $d(x,z)\leq Ct$.
(Namely, if $d(x,z) \geq Ct$ then $d(x,z) \leq (1+2\delta)|P(x)-P(z)|$ by Lemma~\ref{ledh}.)
The claim now follows from \eqref{apu6}.
\end{proof}

Now \eqref{apu1}, \eqref{apu2}, Lemma~\ref{lesg1}, \eqref{apu3}, Lemma~\ref{lesg2}, Lemma~\ref{lesg3} and \eqref{Vpt} give
(by choosing $\varepsilon < 1$ and $K_0$ large enough)
\[ \gamma(p,t) \leq C\varepsilon^{-3k}\beta_1(z,K_0t) + C\sqrt{\varepsilon}t^{-k-1}\sum_{i\in I(p,t)} d(R_i)^{k+1} \]
and further by the regularity
\begin{align}
\label{apu7}
  \gamma(p,t)^2 \leq C\varepsilon^{-6k}t^{-k}\int_{B(z(p,t),t)} \beta_1(z,K_0t)^2\, d\mu z + C\varepsilon t^{-2(k+1)}\biggl(\sum_{i\in I(p,t)} d(R_i)^{k+1}\biggr)^2
\end{align}
for some constant~$C=C(K_0)$.

If $p\in U_1$, $H(p)/60<t\leq T$ and $z\in B(z(p,t),t)\cap E$ then
$z\in K_0Q(\ST)$ and $|p-P(z)| \leq |p-P(z(p,t))| + |P(z(p,t))-P(z)| \leq Ct$.
We also have $h(z) \leq Ct$.
(Namely, choose $\tilde{u}\in P^{-1}(\{p\})$ with $h(\tilde{u})\leq H(p)+t$.
Then let $u\in Q(\ST)$ be with $d(u,\tilde{u}) \leq h(\tilde{u}) + t$.
If $d(z,u)>h(u)$ then $d(z,u)\leq 2|P(z)-P(u)| \leq 2(|P(z)-p|+d(u,\tilde{u}))\leq Ct$ by Lemma~\ref{ledh}.
In any case $h(z)\leq h(u)+d(z,u) \leq Ct$.)
Thus
\begin{align}
\label{apu8}
\begin{split}
  &\int_{U_1}\int_{H(p)/60}^T t^{-k}\int_{B(z(p,t),t)} \beta_1(z,K_0t)^2\, d\mu z\, \frac{dt}{t}\, dp  \\
  &\leq \int_{K_0Q(\ST)}\int_{h(z)/C}^T t^{-k}\biggl(\int_{B^k(P(z),Ct)}\, dp\biggr) \beta_1(z,K_0t)^2\, \frac{dt}{t}\, d\mu z  \\
  &\leq C\int_{K_0Q(\ST)}\int_{h(z)/C}^T \beta_1(z,K_0t)^2\, \frac{dt}{t}\, d\mu z
\end{split}
\end{align}
Since $d(p,R_i)\leq t$ and $d(R_i)\leq Ct$ if $p\in U_1$, $H(p)/60<t\leq T$ and $i\in I(p,t)$ (as mentioned after \eqref{apu6})
\begin{align}
\label{apu9}
\begin{split}
  &C^{-1}\int_{U_1}\int_{H(p)/60}^T t^{-2(k+1)}\biggl(\sum_{i\in I(p,t)} d(R_i)^{k+1}\biggr)^2\, \frac{dt}{t}\, dp  \\
  &\leq \int_{U_1}\int_{H(p)/60}^T \sum_{i\in I(p,t)} d(R_i)^{k+1}\, \frac{dt}{t^{k+2}}\, dp
  \leq \sum_{i\in I_0} d(R_i)^{k+1} \int_{d(R_i)/C}^T\int_{R_i(t)}\, dp\, \frac{dt}{t^{k+2}}  \\
  &\leq \sum_{i\in I_0} d(R_i)^{k+1} \int_{d(R_i)/C}^\infty (d(R_i)+2t)^k\, \frac{dt}{t^{k+2}}
  \leq C\sum_{i\in I_0} d(R_i)^k \leq C\mu(Q(\ST)),
\end{split}
\end{align}
where the last constant $C$ depends on $K_0$.
The last inequality follows from \eqref{D} because $d(R_i)\leq K_0d(Q(\ST))/20$ by \eqref{H} and \eqref{defR}
and therefore $R_i\subset B^k(P(x_0),2K_0d(Q(\ST)))$ for any $i\in I_0$.
From \eqref{apu7}, \eqref{apu8} and \eqref{apu9} one now gets
\begin{align*}
  \int_{U_1}\int_{H(p)/60}^T \gamma(p,t)^2\, \frac{dt}{t}\, dp
  \leq C\varepsilon \mu(Q(\ST))
  + C\varepsilon^{-6k}\int_{K_0Q(\ST)}\int_{h(z)/C}^T \beta_1(z,K_0t)^2\, \frac{dt}{t}\, d\mu z.
\end{align*}
Combining this with Lemma~\ref{legapu} we get Lemma~\ref{leg} because
$U_1 \subset P(Z) \cup \bigcup_{i\in I_1} R_i$, $H(P(Z))=\{0\}$ and $60t>H(p)$ by \eqref{RH} whenever $p\in R_i$ with $d(R_i)<t$ and $i\in I_0$.

\section{Estimate for $\ST\in\F_3$}
\label{secF3}

In this section we use the same assumptions and notations as in Section~\ref{secF}.
We follow \cite[Sections 14 and 11]{MR1113517} or \cite[Section 5]{MR1709304} and prove the next lemma.

\begin{lemma}
\label{leF3}
Let $\ST\in\F_3$. Then
\[ \int_{K_0Q(\ST)}\int_{h(x)/K_0}^{K_0d(Q(\ST))} \beta_1(x,K_0t)^2\, \frac{dt}{t}\, d\mu x > \varepsilon^{6k+1}\mu(Q(\ST)). \]
\end{lemma}

Fix $\ST\in\F_3$ and suppose to the contrary that the claim is not true for $\ST$.
Since the translations and the rotations are isometries we can assume that $V_{Q(\ST)}=X_k$ (see \ref{remProttr}).
We use the same notations as in Section~\ref{secS}.
By Lemma~\ref{leg}
\begin{align}
\label{AT}
  \int_0^{K_0d(Q(\ST))/2} \int_{U_1}\gamma(p,t)^2\, dp\, \frac{dt}{t}\leq C(K_0)\varepsilon\mu(Q(\ST)).
\end{align}
Let $\nu:\R^k\to\R$ be a radial $C^\infty$ function supported in $B^k(0,1)$ such that
$\int_{\R^k} f\nu = 0$
for any affine function $f:\R^k\to\R$ and
\[ \int_0^\infty |\hat\nu(tp)|^2\, \frac{dt}{t} = 1\]
for all $p\in\R^k\backslash\{0\}$. Denote $\nu_t(p)=t^{-k}\nu(t^{-1}p)$ for any $p\in\R^k$ and $t>0$.
Using Calder\'on's formula one can write
\[ g(p) = \int_0^\infty (\nu_t*\nu_t*g)(p)\, \frac{dt}{t} \]
for any $p\in\R^k$. (Notice that the above integral exists and depends continuously on $p$, because $g$ is
Lipschitz and has compact support.)
Set $L=K_0d(Q(\ST))/5$ and write $g=g_1+g_2$, where $g_1,g_2:\R^k\to\R^{2n-k}$ are defined by
\begin{align*}
  g_1(p) &= \int_L^\infty (\nu_t*\nu_t*g)(p)\, \frac{dt}{t} + \int_0^L (\nu_t*(\chi_{\R^k\backslash U_1}\cdot(\nu_t*g)))(p)\, \frac{dt}{t}, \\
  g_2(p) &= \int_0^L (\nu_t*(\chi_{U_1}\cdot(\nu_t*g)))(p)\, \frac{dt}{t}.
\end{align*}

\begin{lemma}
\label{leg1}
There is a constant $C$ such that $|\partial_jg_1(p)| \leq C\sqrt{\delta}$ and $|\partial_i\partial_jg_1(p)| \leq C\sqrt{\delta}/L$
for any $i,j\in\{1,\dotsc,k\}$ and $p\in U_2$.
\end{lemma}

\begin{proof}
We first notice that
\begin{align}
\label{apu10}
  \int_0^L (\nu_t*(\chi_{\R^k\backslash U_1}\cdot(\nu_t*g)))(p)\, \frac{dt}{t} = 0
\end{align}
for all $p\in U_{9/5} \supset U_2$.
Set
\[ \varphi(q) = \int_L^\infty (\nu_t*\nu_t)(q)\, \frac{dt}{t} \]
for all $q\in\R^k$.
By \eqref{apu10} one has $g_1(p)=(\varphi*g)(p)$ for all $p\in U_{9/5}$.
Fix $i,j\in\{1,\dotsc,k\}$ and $p\in U_2$.
Since $|\nu_t*\nu_t| \leq Ct^{-k}$ for any $t>0$, one has $|\varphi| \leq CL^{-k}$.
Further $|\nabla\varphi| \leq CL^{-k-1}$.
(Here $C$ depends on $\nu$.)
Particularly
\begin{align}
\label{apu11}
  \int_{U_{-1}}|\varphi(p-q)|\, dq \leq C\qquad\text{and}\qquad\int_{U_{-1}}|\partial_i\varphi(p-q)|\, dq \leq \frac{C}{L}.
\end{align}
Since $\varphi$ is bounded and $g$ has compact support, Lemma~\ref{leglip} and the dominated convergence
give $\partial_jg_1(p) = (\varphi*\partial_j g)(p)$. (Notice that $g$ is differentiable almost everywhere by Rademacher's theorem.)
Thus, since $\partial_i\varphi$ is bounded and $\partial_jg$ is compactly supported and bounded,
we further have $\partial_i\partial_jg_1(p) = (\partial_i\varphi*\partial_j g)(p)$.
Since $\partial_j g$ is supported in $U_{-1}$, the claim now follows from Lemma~\ref{leglip} and \eqref{apu11}.
\end{proof}

Define $g_{2,m}$ for any $m\in\N$ by setting
\begin{align*}
  g_{2,m}(p) &= \int_{1/m}^L (\nu_t*(\chi_{U_1}\cdot(\nu_t*g)))(p)\, \frac{dt}{t}
\end{align*}
for all $p\in\R^k$. Then $g_{2,m}\to g_2$ uniformly as $m\to\infty$. We also have that $g_{2,m}\to g_2$ in $L^2$
because $g_2$ is bounded and $\spt g_{2,m} \subset U_0$ for all $m$.
Now
\begin{align*}
  \partial_jg_{2,m}(p) = \int_{1/m}^L (\partial_j\nu_t*(\chi_{U_1}\cdot(\nu_t*g)))(p)\, \frac{dt}{t}
\end{align*}
for any $p\in\R^k$, $j\in\{1,\dotsc,k\}$ and $m\in\N$.
Using this we find a constant~$C=C(\nu)$ such that
\begin{align}
\label{apu12}
  \int_{\R^k} |\partial_jg_{2,m}(p)|^2\, dp \leq C\int_{1/m}^L \int_{U_1}\gamma(p,t)^2\, dp\, \frac{dt}{t}
\end{align}
for any $j\in\{1,\dotsc,k\}$ and $m\in\N$ (see \cite{MR1113517} or one dimensional case \cite[page~863]{MR1709304}).
Particularly $(g_{2,m})_m$ is a bounded sequence in the Sobolev space $W^{1,2}$ (by \eqref{AT}) and
a subsequence of $(\partial_jg_{2,m})_m$  converges weakly in $L^2$ to $\partial_jg_2$.
Thus by \eqref{apu12} and \eqref{AT}
\begin{align}
\label{dleqmu}
  \int_{\R^k} |\partial_jg_2(p)|^2\, dp \leq C(K_0)\varepsilon\mu(Q(\ST))
\end{align}
for any $j\in\{1,\dotsc,k\}$.

Define a function $N:\R^k\to\R$ by setting
\[ N(p) = \sup_B\frac{m_B(|g_2-m_B(g_2)|)}{d(B)}, \]
where the supremum is taken over all balls $B$ in $\R^k$ containing $p$ and having (positive) radius at most $L$.
Here we use the notation $m_B(f)=\dashint_B f$ for locally integrable functions $f:\R^k\to\R$.
From Poincar\'e's inequality and the Hardy-Littlewood maximal inequality one now gets
(see \cite[page~75]{MR1113517} or one dimensional case \cite[Lemma~5.3]{MR1709304})
\begin{align}
\label{Nleqd}
  \int_{\R^k} N(p)^2\, dp \leq C\max_j\int_{\R^k} |\partial_jg_2(p)|^2\, dp.
\end{align}
Since $g_2|_{U_2}$ is $C\sqrt{\delta}$-Lipschitz by Lemmas~\ref{leglip} and \ref{leg1}, one gets
(see \cite[Lemma~11.8]{MR1113517})
that for any closed ball $B\subset U_2$
\begin{align}
\label{apu13}
  \sup_{p\in B} |g_2(p)-m_B(g_2)| \leq C\delta^{\frac{k}{2(k+1)}}d(B)N(q)^{\frac{1}{k+1}}
\end{align}
whenever $q\in B$.
Set $F=\{\, q\in U_3\, :\, N(q)^3 \leq \varepsilon\, \}$.
Using \eqref{apu13}, Lemma~\ref{leg1} and Taylor's theorem one gets (see \cite[Lemma~11.9]{MR1113517}) that
\begin{align}
\label{apu14}
  \sup_{p\in B^k(p_0,r)} |g(p)-g(p_0)-Dg_1(p_0)(p-p_0)| \leq C\delta^{\frac{k}{2(k+1)}}\varepsilon^{\frac{1}{3(k+1)}}r + C\sqrt{\delta}L^{-1}r^2
\end{align}
whenever $r\leq L/4$ and $B^k(p_0,r)\cap F\neq\emptyset$.
For any $p\in U_2$ let $\Delta_p\subset\R^{2n}$ be the $k$-plane which is the graph of the affine function
$q\mapsto g(p)+Dg_1(p)(q-p)$.

\begin{lemma}
\label{lem2F}
If $Q\in m_2(\ST)$ then $d(P(Q),F) > d(Q)$.
\end{lemma}

\begin{proof}
Suppose to the contrary that $Q\in m_2(\ST)$ with $d(P(Q),F) \leq d(Q)$.
Since $\CH(Q) \subset \G$ by definition of $m_2(\ST)$, we get by (F7) a contradiction $Q\not\in\min(\ST)$ by showing that
$\kulma{V_R}{V_{Q(\ST)}} \leq 1 + \delta$ for for all $R\in\CH(Q)$.
For that reason, fix $R \in \CH(Q)$.  
If now $2Kd(R) \geq d(Q(\ST))$ then $\kulma{V_R}{V_{Q(\ST)}} \leq 1 + \varepsilon$ by Lemma~\ref{leQP}.
Thus we may assume that $2Kd(R) < D^2\alpha d(Q(\ST))$.
Pick $x\in Q$ and set $r= 3d(Q)$.
Now $r < L/K \leq L/4$ (by \eqref{D} choosing $K_0$ and $K$ large enough) and $B^k(P(x),r) \cap F \neq\emptyset$.
By this, Lemma~\ref{legappr} and \eqref{apu14}
\begin{align}
\label{apu15}
\begin{split}
  d_e(y',\Delta_{P(x)}) &\leq | P^\bot(y) - g(P(x)) - Dg_1(P(x))(P(y)-P(x))| \\
  &\leq C\sqrt{\varepsilon}h(y) + C\delta^{\frac{k}{2(k+1)}}\varepsilon^{\frac{1}{3(k+1)}}r + C\sqrt{\delta}K^{-1}r
\end{split}
\end{align}
for any $y\in 2Q$.
Let $y_0,\dotsc,y_k\in V_{R} \cap Q(\varepsilon^2 d(R))$ be as in Lemma~\ref{leL} (recalling that $R\in\G$).
Then by \eqref{dV}, \eqref{D}  and \eqref{apu15}
\[ d_e(y_j',\Delta_{P(x)}) \leq C\left( \sqrt{\varepsilon} + \delta^{\frac{k}{2(k+1)}}\varepsilon^{\frac{1}{3(k+1)}} + \sqrt{\delta}K^{-1} \right)d(R) \]
for any $j\in\{1,\dotsc,k\}$. 
Further
$d_e(y_{i+1}',L_i') > cd(R)$ for all $i\in\{0,\dotsc,k-1\}$,
where $L_i$ and $c$ are as in Lemma~\ref{leL}.
Thus the euclidean angle between $V_R'$ and $\Delta_{P(x)}$ is less than $\delta/9$ by taking $\varepsilon$ small enough and $K$ large enough depending on $\delta$.
Let $Q^*$ be the minimal cube in $\ST$ such that $Q\subset Q^*$ and $2Kd(Q^*) \geq d(Q(\ST))$.
Then $2Kd(Q^*) < D^2\alpha d(Q(\ST))$ (by \eqref{D}) and by the above argument the angle between $V_{Q*}'$ and $\Delta_{P(x)}$ is also
less than $\delta/9$.
Now $\kulma{V_{Q^*}}{V_{Q(\ST)}} \leq 1+\varepsilon$ by Lemma~\ref{leQP}. Choosing $\varepsilon/\delta$ small the euclidean angle between
$V_{Q^*}'$ and $V_{Q(\ST)}'$ is less than $\delta/9$ (by Lemma~\ref{lekulmat}).
Thus the angle between $V_R'$ and $V_{Q(\ST)}'$ is less than $\delta/3$ and so $\kulma{V_R}{V_{Q(\ST)}} \leq 1 + \delta$ (by choosing $\delta$ small).
\end{proof}

For each $Q\in m_2(\ST)$ pick $x_Q\in Q$.
By the $5r$-covering lemma we find $\mathcal{T}\subset m_2(\ST)$ such that
the balls $B(x_Q,3d(Q))$, $Q\in\mathcal{T}$, are disjoint and
\[ G := \bigcup_{Q\in m_2(\ST)} Q \subset \bigcup_{Q\in\mathcal{T}} B(x_Q,15d(Q)). \]
Since $d(x_Q,x_{R}) > 3\max\{d(Q),d(R)\} \geq h(x_Q)$ for any distinct $Q,R\in\mathcal{T}$, Lemma~\ref{ledh} gives
(by choosing $\delta$ small) that the balls $B^k(P(x_Q),d(Q))$, $Q\in\mathcal{T}$, are also disjoint.
Further $B^k(P(x_Q),d(Q)) \subset U_3\backslash F$ for any $Q\in m_2(\ST)$ by Lemma~\ref{lem2F}.
Hence by \eqref{Nleqd} and \eqref{dleqmu}
\begin{align*}
  \mu(G) &\leq \sum_{Q\in\mathcal{T}}\mu(B(x_Q,15d(Q))) \leq 15^kC_E\sum_{Q\in\mathcal{T}}d(Q)^k \\
  &\leq C\LEK\biggl(\bigcup_{Q\in\mathcal{T}}B^k(P(x_Q),d(Q))\biggr) \leq C\LEK(U_3\backslash F) \\
  &\leq C\varepsilon^{-2/3}\int_{U_3\backslash F} N(p)^2\, dp \leq C(K_0)\varepsilon^{1/3}\mu(Q(\ST)).
\end{align*}
Choosing $\varepsilon$ small enough this means that $\ST\not\in\F_3$ which is a contradiction.

\section{End of the proof}
\label{end}

In this section we follow \cite[Sections 12 and 16]{MR1113517} and use the same assumptions and notations as in Section~\ref{secbd}
and assume further that \eqref{oletus} is satisfied.
Now the assumptions of Sections~\ref{secF} and \ref{secF3} are also satisfied (see Section~\ref{secbd}).
From this on the constants $C$ may depend without special mention on $\varepsilon$, $K$, $\delta$ or $K_0$.

\begin{lemma}
\label{leF}
  There is a constant $C$ such that
  \[ \sum_{\substack{\ST\in\F \\ Q(\ST)\subset R}} \mu(Q(\ST)) \leq C\mu(R)\qquad\text{for all $R\in\Delta$}. \]
\end{lemma}

\begin{proof}
For any $\ST\in\F$ denote
\[ E_\ST = \{\, (x,t) \in K_0Q(\ST) \times ]0,K_0d(Q(\ST))[\, : \, h_\ST(x) < K_0t\, \}. \]
Suppose for a while that $\ST\in\F$ and $(x,t)\in E_{\ST}$. Then
$d(x,Q)+d(Q) < K_0t < K_0^2d(Q(\ST))$ for some $Q\in\ST$. Let $Q^*$ be the minimal cube in $\ST$ such that $Q\subset Q^*$
and $K_0d(Q^*) > t$. Then (by \eqref{D}) $K_0d(Q^*) \leq D^2\alpha t$ and $C\mu(Q^*) > t^k$.
Since trivially $d(x,Q^*) \leq d(x,Q) < K_0t$ we conclude by the regularity that there is a constant $C$ such that
\begin{align}
\label{apu16}
  \sum_{\ST\in\F} \chi_{E_\ST}(x,t) \leq C\qquad\text{for all $(x,t) \in E\times\R$}.
\end{align}
Using Lemma~\ref{leF3}, \eqref{apu16}, \eqref{oletus}\label{ol} and \eqref{D} one gets for any $R\in\Delta$
\begin{align*}
  \sum_{\substack{\ST\in\F_3 \\ Q(\ST)\subset R}} \mu(Q(\ST))
  &< \varepsilon^{-6k-1} \sum_{\substack{\ST\in\F_3 \\ Q(\ST)\subset R}}
  \int_{K_0Q(\ST)}\int_{h_\ST(x)/K_0}^{K_0d(Q(\ST))} \beta_1(x,K_0t)^2\, \frac{dt}{t}\, d\mu x  \\
  &\leq C\int_0^{K_0d(R)}\int_{K_0R} \beta_1(x,K_0t)^2\, d\mu x\, \frac{dt}{t}
  \leq C\mu(R).
\end{align*}
The claim now follows from Lemma~\ref{leF1F2}.
\end{proof}

Theorem~\ref{th} follows from the following lemma.

\begin{lemma}
  For any $\eta>0$ there is $C>0$ such that for all $z\in E$ and $r>0$ there is $F\subset\HE^n$ and
  a $C$-bilipschitz mapping $f:F\to\R^k$ such that
  $\mu(B(z,r)\backslash F) \leq \eta r^k$.
\end{lemma}

\begin{proof}
Let $\eta>0$, $z\in E$ and $r\in\R$ with $0<r\leq d(E)$. Let $m_0\in\ZZ$ be such that $D\alpha^{m_0-1} < r\leq D\alpha^{m_0}$.
Set
\begin{align*}
  \mathcal{R}_0 &= \{\, Q\in\Delta_{m_0}\, :\, Q\cap B(z,r)\neq\emptyset\, \}, \\
  \tilde{\F} &= \left\{\, \ST \cap \{ Q\, :\, Q\subset R\}\, :\, \text{$\ST\in\F$, $R\in\mathcal{R}_0$}\, \right\} \backslash\left\{\emptyset\right\}.
\end{align*}
Further let
\[ \tilde{\Delta} = \bigcup_{j=-\infty}^{m_0} \tilde{\Delta}_j\qquad\text{and}\qquad\tilde{\G} = \G\cap\tilde{\Delta}, \]
where
\[ \tilde{\Delta}_j = \biggl\{\, Q\in\Delta_j\, :\, Q\subset\bigcup_{R\in\mathcal{R}_0} R\, \biggr\}. \]
One easily sees that (F1), (F2) and Lemma~\ref{leF} remain valid if $\Delta_j$, $\Delta$, $\G$ and $\F$
are replaced by $\tilde{\Delta}_j$, $\tilde{\Delta}$, $\tilde{\G}$ and $\tilde{\F}$.
(For Lemma~\ref{leF} this is because the new maximal cubes~$Q(\ST)$, $\ST\in\tilde{\F}\backslash\F$, belong to $\mathcal{R}_0$.)

For any $Q\in\tilde{\Delta}$ denote
\[ \sigma(Q) = \{\, x\in Q\, :\, d(x,E\backslash Q) \leq \tau\alpha^{j_Q}\, \}, \]
where $\tau$ is a small positive constant fixed later.
Let $\T=\T_1\cup\T_2\cup\T_3$, where $\T_1=\tilde{\Delta}\backslash\G$,
$\T_2=\{\, Q(\ST)\, :\, \ST\in\tilde{\F}\, \}$ and $\T_3=\bigcup_{\ST\in\tilde{\F}} \min(\ST)$.
For any $Q\in\T$ set $\ell(Q) = \#\{\, R\in\T\, :\, R\neq Q\subset R\, \}$.
Define
\[ F = \biggl( B(z,r) \cap \bigcap_{j\in\Z}\bigcup_{Q\in\Delta_j} Q \biggr) \backslash \left( F_1\cup F_2 \right), \]
where
\[ F_1 = \bigcup_{Q\in\T} \sigma(Q)\qquad\text{and}\qquad F_2=\bigcup_{\substack{Q\in\T \\ \ell(Q) > M}} Q. \]
Using \eqref{D1}, (F1), Lemma~\ref{leF} and \eqref{D6} and choosing $\tau$ small enough and $M$ large enough depending on $\eta$
one gets $\mu(B(z,r)\backslash F)\leq\mu(F_1\cup F_2) \leq \eta r^k$ (see \cite[pages~102--103]{MR1113517}).

For the definition on $f$ we first define a map $t:\T\to \mathcal{P}(\R^k)$ recursively as follows.
First, for each $Q\in\mathcal{R}_0 = \tilde{\Delta}_{m_0}$ let $t(Q)$ be a cube in $\R^ k$ with side length $\alpha^{j_Q}$.
Since $\#\mathcal{R}_0 \leq 2^ kC_ED^{k+1}$ by the regularity and \eqref{D}, one can choose the cubes $t(Q)$ so that
\begin{align}
\label{talku}
  \alpha^{m_0}\leq d(t(Q_1),t(Q_2)) \leq C\alpha^{m_0}
\end{align}
for any distinct $Q_1,Q_2\in\mathcal{R}_0$ where $C$ is a constant (depending only on $C_E$, $D$ and $k$).

Let $Q\in\T$. Assume by recursion that a cube $t(Q)\subset\R^k$ has already been defined such that
\begin{align}
\label{rec}
  l(t(Q)) = c_1^{\ell(Q)}\alpha^{j_Q},
\end{align}
where $l(G)=d(G)/\sqrt{k}$ for $G\subset\R^k$ and $c_1>0$ is a small constant to be chosen later.
Assume first that $Q\in\T_1\cup\T_3$. Then $\ell(R)=\ell(Q)+1$ for all $R\in\CH(Q)\subset\T_1\cup\T_2$.
Since further $j_R<j_Q$ for all $R\in\CH(Q)$ and $\#\CH(Q)\leq D^2\alpha^k$ (by \eqref{D}),
one can choose by \eqref{rec} the cubes $t(R)$, $R\in\CH(Q)$, such that
\begin{align}
\label{apu17}
   l(t(R)) &= c_1^{\ell(R)}\alpha^{j_R}, \\   
\label{apu18}
   t(R) &\subset t(Q), \\
\label{apu19}
   d(t(R),t(R_1)) &\geq c_1^{\ell(Q)+1}\alpha^{j_Q}
\end{align}
for all distinct $R,R_1\in\CH(Q)$ provided $c_1$ is small enough (depending on $D$, $\alpha$ and $k$).

Assume now that $Q\in\T_2$ i.e. $Q=Q(\ST)$ for some $\ST\in\tilde{\F}$.
Denote $W_Q=V_{Q(\ST_0)}$ where $\ST_0\in\F$ is such that $\ST=\ST_0\cap\{ Q\, :\, Q\subset Q_0\}$ for some $Q_0\in\mathcal{R}_0$.
By the 1-Lipschitzness of $P_{W_Q}$, \eqref{D4}, \eqref{dV} and \eqref{rec} there is a function $\phi_Q:W_Q\to\R^k$ such that
\begin{align}
\label{phi1}
  &\phi_Q(P_{W_Q}(Q)) \subset 2^{-1}t(Q),  \\
\label{phi2}
  &|\phi_Q(p)-\phi_Q(q)| = \frac{1}{4D}c_1^{\ell(Q)}d(p,q)
\end{align}
for all $p,q\in W_Q$.
Here $\lambda G = \{\, x\in G\, :\, d(x,\R^k\backslash G) \geq (1-\lambda)l(G)/2\, \}$ for $\lambda\in\R$ and a cube $G\subset\R^k$.
For any $R\in\min(\ST)\backslash\{Q\}$ let $t(R)$ be a cube satisfying \eqref{apu17} and centered at $\phi_Q(P_{W_Q}(x_R))$,
where $x_R\in R$ is such that (see \eqref{D5} and \eqref{D4})
\begin{align}
\label{x_R}
  B(x_R,D^{-2}d(R)) \cap E \subset R.
\end{align}
Then \eqref{apu18} holds for $Q$ and any $R\in\min(\ST)$ by \eqref{phi1} and \eqref{rec} provided $2c_1 \leq \alpha$.
Since by \eqref{x_R} and \eqref{D3} $d(x_{R_1},x_{R_2}) \geq D^{-2}d(R_1) \geq D^{-2}h_{\ST_0}(x_{R_1})$ for any distinct $R_1,R_2\in\min(\ST)$,
one has by \eqref{phi2} and Lemma~\ref{ledh}
\begin{align*}
  |\phi_Q(P_{W_Q}(x_{R_1}))-\phi_Q(P_{W_Q}(x_{R_2}))| \geq \frac{c_1^{\ell(Q)}d(x_{R_1},x_{R_2})}{4D(1+2\delta)}
\end{align*}
for any $R_1,R_2\in\min(\ST)$.
Thus by using \eqref{x_R}, \eqref{D} and \eqref{apu17} and choosing $c_1$ small enough
\begin{align}
\label{apug1}
\begin{split}
  d(t(R_1),t(R_2)) &\geq \frac{c_1^{\ell(Q)}d(x_{R_1},x_{R_2})}{8D(1+2\delta)} - \frac{\sqrt{k}(l(t(R_1))+l(t(R_2)))}{2}
  + \frac{c_1^{\ell(Q)}d(R_1,R_2)}{8D(1+2\delta)} \\
  &\geq \left(\frac{1}{8D^3(1+2\delta)} - \sqrt{k}c_1D\right)c_1^{\ell(Q)}\max_{i=1,2}d(R_i)
  + \frac{c_1^{\ell(Q)}d(R_1,R_2)}{8D(1+2\delta)} \\
  &\geq c_1^{\ell(Q)+1}(d(R_1)+d(R_2)+d(R_1,R_2))
\end{split}  
\end{align}
for any distinct $R_1,R_2\in\min(\ST)$.
By \eqref{phi2} and the 1-Lipschitzness of $P_{W_Q}$ also
\begin{align}
\label{apug2}
\begin{split}
  d(t(R_1),t(R_2)) \leq \frac{c_1^{\ell(Q)}d(x_{R_1},x_{R_2})}{4D} \leq d(R_1)+d(R_2)+d(R_1,R_2)
\end{split}  
\end{align}
for any $R_1,R_2\in\min(\ST)$.

For any $x\in F$ let $Q(x)$ be the smallest cube in $\T$ which contains $x$.
Then $Q(x)\in\T_2$ and one can define $f:F\to\R^k$ by setting
\[ f(x) = \phi_{Q(x)}(P_{W_{Q(x)}}(x)) \]
for all $x\in F$.
From this on let $x,y\in F$ be distinct.
Suppose first that $x,y\in Q$ for some $Q\in\T_2$. Let $Q = Q(\ST)$ be the smallest such cube.

Assume very first that there are distinct $R_1,R_2\in\min(\ST)$ with $x\in R_1$ and $y\in R_2$.
Since $x,y\not\in F_1$ one has by \eqref{D4}
\begin{align}
\label{apu20}
  d(x,y) \geq \frac{\tau(d(R_1) + d(R_2))}{3D} +  \frac{d(R_1,R_2)}{3}.
\end{align}  
Since $f(x)\in t(R_1)$ and $f(y)\in t(R_2)$ by definition of $f$, \eqref{phi1} and \eqref{apu18}, one gets
by using \eqref{apu17}, \eqref{D}, \eqref{apug2} and \eqref{apu20}
\begin{align*}
  |f(x)-f(y)| &\leq d(t(R_1))+d(t(R_2))+d(t(R_1),t(R_2))  \\
  &\leq D\sqrt{k}c_1\left( d(R_1)+d(R_2)\right) + d(R_1)+d(R_2)+d(R_1,R_2)  \\
  &\leq Cd(x,y).
\end{align*}
By \eqref{apug1} one also has
\begin{align*}
  |f(x)-f(y)| \geq d(t(R_1),t(R_2)) \geq c_1^{\ell(Q)+1}(d(R_1)+d(R_2)+d(R_1,R_2)) \geq c_1^Md(x,y).
\end{align*}

Assume now that $y\in R_2\in\min(\ST)$ and $x\not\in R$ for all $R\in\min(\ST)$.
Since now $f(x) = \phi_Q(P_{W_Q}(x))$, the argument used to establish \eqref{apug1} and \eqref{apug2} also gives
\begin{align*}
  c_1^{\ell(Q)+1}(d(R_2)+d(x,R_2)) \leq d(f(x),t(R_2)) \leq d(R_2)+d(x,R_2).
\end{align*}
Since further
\begin{align*}
  \frac{\tau d(R_2)}{2D} +  \frac{d(x,R_2)}{2} \leq d(x,y) \leq d(R_2) + d(x,R_2),
\end{align*}
and $f(y)\in t(R_2)$, one has
\begin{align*}
  c_1^Md(x,y) \leq |f(x)-f(y)| \leq Cd(x,y).
\end{align*}

If $x\not\in R$ and $y\not\in R$ for all $R\in\min(\ST)$, then $f(x) = \phi_Q(P_{W_Q}(x))$ and $f(y) = \phi_Q(P_{W_Q}(y))$.
In this case \eqref{phi2}, Lemma~\ref{ledh} and the 1-Lipschitzness of $P_{W_Q}$ give directly
\begin{align*}
  \frac{c_1^M}{4D(1+2\delta)}d(x,y) \leq |f(x)-f(y)| \leq \frac{1}{4D}d(x,y).
\end{align*}
Let now $Q_1$ be the largest cube in $\tilde{\Delta}$ which contains $x$ but not $y$, and denote $Q_0=O(Q_1)$.

Assume that $x,y\in R\in\min(\ST)$ (and that $Q(\ST)$ is still the smallest cube in $\T_2$ with $x,y\in Q$).
Then necessarily $Q_0=R\in\T_3$ or $Q_0\in\T_1$, because otherwise $x,y \in Q_0\not\in\ST\cup\T_1$ which contradicts the minimality of $Q(\ST)$.
Now $y\in Q_2$ for some $Q_2\in\CH(Q_0)\backslash\{ Q_1\}$ and $Q_1,Q_2\in\T$.
As before, by \eqref{D4} and the definition of $F$
\begin{align*}
  D^{-1}\tau d(Q_1) \leq d(x,y) \leq d(Q_0).
\end{align*}
Further $f(x)\in t(Q_1) \subset t(Q_0)$ and $f(y)\in t(Q_2) \subset t(Q_0)$ (by definition of $f$, \eqref{phi1} and \eqref{apu18}).
Thus by \eqref{rec} (and \eqref{apu17}) and \eqref{D}
\begin{align}
\label{apu21}
  |f(x)-f(y)| \leq d(t(Q_0)) \leq c_1^{\ell(Q_0)}\sqrt{k}\alpha^{j_{Q_0}} \leq \sqrt{k}D^2\alpha\tau^{-1}d(x,y),
\end{align}
and by \eqref{apu19} and \eqref{D4}
\begin{align}
\label{apu22}
  |f(x)-f(y)| \geq c_1^{\ell(Q_0)+1}\alpha^{j_{Q_0}} \geq c_1^MD^{-1}d(Q_0) \geq c_1^MD^{-1}d(x,y).
\end{align}

Finally assume that there does not exist $Q\in\T_2$ with $x,y\in Q$.
Then $Q_1\in\mathcal{R}_0$ or $Q_0\in\T_1$.
In the latter case \eqref{apu21} and \eqref{apu22} are obtained as above.
In the former case
\begin{align*}
  D^{-2}\tau\alpha^{m_0} \leq D^{-1}\tau d(Q_1) \leq d(x,y) \leq 2r \leq 2D\alpha^{m_0}
\end{align*}
by \eqref{D} and the definition on $F$.
Further
\begin{align*}
  \alpha^{m_0} \leq |f(x)-f(y)| \leq (C+2)\alpha^{m_0}
\end{align*}
by \eqref{talku} and \eqref{rec}.
(This is again because $f(x)\in t(Q_1)$ and $f(y)\in t(Q_2)$, where $Q_2\in\mathcal{R}_0$ is such that $y\in Q_2$.)
This gives
\begin{align*}
  \frac{1}{2D}d(x,y) \leq |f(x)-f(y)| \leq \frac{D^2(C+2)}{\tau}d(x,y).
\end{align*}
\end{proof}

\bibliographystyle{plain}
\bibliography{viitteet}

\end{document}